\documentclass{amsart}
\usepackage[final]{epsfig}
\usepackage{graphics}
\usepackage{float}
\usepackage{amsmath}
\usepackage{amsfonts}
\usepackage{latexsym}
\usepackage{amssymb}
\usepackage{graphicx}
\usepackage{url}
\usepackage{epstopdf}
\usepackage{verbatim}
\usepackage[margin=1in]{geometry}
\usepackage{amsthm}
\usepackage{tikz}
\usepackage{caption}
\usepackage{tikz-cd}
\usepackage{enumitem}

\makeatletter
\newtheorem*{rep@theorem}{\rep@title}

\definecolor{orange}{rgb}{1,0.5,0}

\newcommand{\To}{\longrightarrow}

\newcommand{\newreptheorem}[2]{%
\newenvironment{rep#1}[1]{%
 \def\rep@title{#2 \ref{##1}}%
 \begin{rep@theorem}}%
 {\end{rep@theorem}}}
\makeatother

\newtheorem{lemma}{Lemma}[section]
\newtheorem{proposition}[lemma]{Proposition}

\newtheorem{remark}[lemma]{Remark}
\newtheorem{example}[lemma]{Example}
\newtheorem{theorem}[lemma]{Theorem}
\newtheorem{definition}[lemma]{Definition}
\newtheorem{corollary}[lemma]{Corollary}

\newtheorem*{theorem*}{Theorem}
\newreptheorem{theorem}{Theorem}

\newcommand{\PP}{{\mathcal P}}
\newcommand{\PPP}{{\mathsf{Q}}}
\newcommand{\Fl}{{\mathsf{Fl}}}
\newcommand{\C}{{\mathbb C}}

\newcommand{\F}{{\mathbb F}}
\newcommand{\N}{{\mathbb N}}
\newcommand{\R}{{\mathbb R}}
\newcommand{\Z}{{\mathbb Z}}

\newcommand{\ttt}{{\mathfrak t}}

\newcommand{\B}{\mathcal{B}}

\DeclareMathOperator{\SL}{SL}
\DeclareMathOperator{\Sl}{SL}
\DeclareMathOperator{\Int}{int}

\DeclareMathOperator{\convex}{Conv}
\DeclareMathOperator{\conv}{Conv}

\DeclareMathOperator{\cone}{Cone}
\DeclareMathOperator{\Perm}{Perm}
\DeclareMathOperator{\Patt}{Patt}

\DeclareMathOperator{\M}{\mathcal M}

\def\v{\mathbf{v}}
\def\u{\mathbf{u}}

\begin{document}

\title{Bruhat interval polytopes}

\author{E. Tsukerman and L. Williams}
\date{\today}
\thanks{The first author was supported by 
a NSF Graduate Research Fellowship under Grant No. DGE 1106400.  The
second author was
partially supported by an NSF CAREER award DMS-1049513.}
\address{Department of Mathematics, University of California,
Berkeley, CA 94720-3840}
\email{e.tsukerman@berkeley.edu}
\address{Department of Mathematics, University of California,
Berkeley, CA 94720-3840}
\email{williams@math.berkeley.edu}
\dedicatory{Dedicated to Richard Stanley on the occasion of his 70th birthday}
\subjclass[2000]{}

\begin{abstract}

Let $u$ and $v$ be permutations on $n$ letters, with 
$u \leq v$ in Bruhat order.  
A \emph{Bruhat interval polytope} $\PPP_{u,v}$ is the convex
hull of all permutation vectors $z = (z(1), z(2),\dots,z(n))$
with $u \leq z \leq v$.  Note that when $u=e$ and $v=w_0$ are the shortest
and longest elements of the symmetric group, 
$\PPP_{e,w_0}$ is the classical permutohedron.
Bruhat interval polytopes were studied recently
in \cite{KW3} by Kodama and the second author, in the context
of the Toda lattice and 
the moment map on the flag variety.

In this paper we study combinatorial aspects of Bruhat interval polytopes.
For example, we give an inequality description and a dimension formula
for Bruhat interval polytopes,
and prove that every face of a Bruhat interval polytope is a Bruhat 
interval polytope.  A key tool in the proof of the latter statement
is a generalization of the well-known lifting property for Coxeter groups.
Motivated by the relationship between the lifting property and 
$R$-polynomials, we also give a generalization of the standard recurrence
for $R$-polynomials.  Finally, 
we define a more general class of polytopes called
\emph{Bruhat interval polytopes for $G/P$}, which are 
moment map images of (closures of) totally positive cells in $(G/P)_{\geq 0}$,
and are a special class of Coxeter matroid polytopes.
Using tools from total positivity and the Gelfand-Serganova stratification,
we show that the face of any Bruhat interval polytope for $G/P$
is again a Bruhat interval polytope for $G/P$.
\end{abstract}

\maketitle
\setcounter{tocdepth}{1}
\tableofcontents

\section{Introduction}

The classical \emph{permutohedron} is the convex hull of all 
permutation vectors $(z(1),z(2),\dots,z(n)) \in \R^n$ where $z$ is an element
of the symmetric group $S_n$.  It has many beautiful properties:
its edges are in bijection with cover relations in the 
weak Bruhat order; its faces can be described 
explicitly; it is the Minkowski sum of matroid polytopes;
it is the moment map image of the complete flag variety.

The main subject of this paper is a natural generalization of the permutohedron
called a Bruhat interval polytope.  Let $u$ and $v$ be 
permutations in $S_n$, with $u \leq v$ in (strong) Bruhat order.
The \emph{Bruhat interval polytope} 
(or \emph{pairmutohedron}\footnote{
While the name ``Bruhat interval polytope" is descriptive, it is 
unfortunately a bit cumbersome.  At the Stanley 70 conference,
the second author asked the audience for suggestions for 
alternative names.  Russ Woodroofe suggested the name
``pairmutohedron"; additionally, Tricia Hersh suggested the name
``mutohedron" (because a Bruhat interval polytope is a subset
of the permutohedron).})
$\PPP_{u,v}$ is the convex
hull of all permutation vectors $z = (z(1), z(2),\dots,z(n))$
with $u \leq z \leq v$.  Note that when $u=e$ and $v=w_0$ are the shortest
and longest elements of the symmetric group, 
$\PPP_{e,w_0}$ is the classical permutohedron.
Bruhat interval polytopes were recently  
studied in \cite{KW3} by Kodama and the second author, in the context
of the Toda lattice and 
the moment map on the flag variety $\Fl_n$.
A basic fact is that $\PPP_{u,v}$ is the 
moment map image of the Richardson variety
$\mathcal{R}_{u,v}  \subset \Fl_n$.
Moreover, 
$\PPP_{u,v}$ is a Minkowski sum of matroid polytopes
(in fact of \emph{positroid polytopes} \cite{ARW}) \cite{KW3}, 
which implies that 
$\PPP_{u,v}$ is a \emph{generalized permutohedron} (in the sense of 
Postnikov \cite{Postnikov2}).

The goal of this paper is to study combinatorial aspects of Bruhat 
interval polytopes.  We give a dimension formula for Bruhat interval polytopes,
an inequality description of Bruhat interval polytopes,
and prove that every face of a 
Bruhat interval polytope is again a Bruhat interval polytope.  In particular,
each edge corresponds to some edge in the (strong) Bruhat order.  The proof
of our result on faces uses the classical result (due to Edelman \cite{Edelman}
in the case of the symmetric group, and subsequently generalized by 
Proctor \cite{Proctor} and then Bjorner-Wachs \cite{BW}) that 
the order complex of an interval in Bruhat order is homeomorphic to a sphere.  Our proof also 
uses a generalization of the lifting property, which appears to be new and may be of interest
in its own right.  
This Generalized lifting property says that if $u<v$ in $S_n$, then
there exists an \emph{inversion-minimal} transposition $(ik)$ (see Definition \ref{def:inversion-minimal})
such that 
$u \leq v(ik) \lessdot v$ and $u \lessdot u(ik) \leq v$.  
One may compare this with 
the usual lifting property, which  says that if $u<v$ and the simple reflection 
$s_i \in D_r(v) \setminus D_r(u)$ is a right-descent of $v$ but not a right-descent of $u$, then 
$u \leq vs_i \lessdot v$ and $u \lessdot us_i \leq v$. Note  that in general
such a simple reflection $s_i$ need not exist.

The usual lifting property is closely related to the $R$-polynomials $R_{u,v}(q)$.  
Recall that the $R$-polynomials
are used to define Kazhdan-Lusztig polynomials \cite{KL}, and also have an interesting 
geometric interpretation:
the Richardson variety 
$\mathcal{R}_{u,v}$ may be defined over a finite field $\F_q$, and the number of points
it contains is given by the $R$-polynomial $R_{u,v}(q) = \#\mathcal{R}_{u,v}(\F_q)$.
A basic result about the $R$-polynomials is that if $s_i \in D_r(v) \setminus D_r(u)$,
then $R_{u,v}(q) = q R_{us, vs}(q) + (q-1) R_{u,vs}(q)$.  We generalize this result, showing that if 
$t=(ik)$ is inversion-minimal, then 
$R_{u,v}(q) = q R_{ut, vt}(q) + (q-1) R_{u,vt}(q)$.  

Finally we give a generalization of Bruhat interval polytopes in the setting of partial flag varieties
$G/P$.  More specifically, let $G$ be a semisimple simply connected linear algebraic group with torus
$T$, and let $P=P_J$ be a parabolic subgroup of $G$. 
Let $W$ be the Weyl group, $W_J$  the corresponding
parabolic subgroup of $W$, 
and let $\mathfrak{t}$ denote the Lie algebra of $T$.  Let $\rho_J$ be the sum of fundamental weights
corresponding to $J$, so that $G/P$ embeds into $\mathbb{P}(V_{\rho_J})$.  Then given $u \leq v$ in $W$,
where $v \in W^J$ is a minimal-length coset representative in $W/W_J$,
we define the corresponding \emph{Bruhat interval polytope for $G/P$} to be
$$\PPP_{u,v}^J:= \conv\{z \cdot \rho_J \ \vert \ u \leq z \leq v\} \subset \mathfrak{t}^*_{\R}.$$
In the $\Fl_n$ case -- i.e. the case that $G = \SL_n$ and $P$ is the Borel subgroup of upper-triangular matrices --
the polytope $\PPP_{u,v}^J$ is a Bruhat interval polytope as defined earlier.
In the Grassmannian case -- i.e. the case that $G = \SL_n$ and $P$ is a maximal parabolic subgroup --
the Bruhat interval polytopes for $G/P$ are precisely the \emph{positroid polytopes}, which 
were studied recently in \cite{ARW}.
As in the $\Fl_n$ case, Bruhat interval polytopes for $G/P$ have an interpretation in terms of the moment map:
we show that $\PPP_{u,v}^J$ is the moment-map image of the closure of a cell
in Rietsch's cell decomposition of $(G/P)_{\geq 0}$.  It is also the moment-map image of 
the projection to $G/P$ of a Richardson variety.
We also show that the face of a Bruhat interval polytope for $G/P$ is a 
Bruhat interval polytope for $G/P$.  Along the way, we 
build on work of Marsh-Rietsch \cite{MR3} to give an interpretation of 
Rietsch's cell decomposition of $(G/P)_{\geq 0}$ in terms of the Gelfand-Serganova stratification
of $G/P$.  In particular, each cell of $(G/P)_{\geq 0}$ is contained in a Gelfand-Serganova stratum.
In nice cases (for example $G/B$ and the Grassmannian $Gr_{k,n}$)
it follows that Rietsch's cell decomposition is the restriction of 
the Gelfand-Serganova stratification to $(G/P)_{\geq 0}$.

The structure of this paper is as follows.  In Section 
\ref{sec:background} we provide background and terminology for posets, Coxeter groups, permutohedra,
matroid polytopes, and Bruhat interval polytopes.
In Section \ref{sec:lifting} we state and prove the Generalized 
lifting property for the symmetric group.
We then use this result in Section \ref{sec:faces} to prove 
that the face of a Bruhat interval polytope is a Bruhat interval 
polytope.  Section \ref{sec:faces} also provides a dimension formula
for Bruhat interval polytopes, and an inequality description
for Bruhat interval polytopes.
In Section \ref{sec:R} we give a generalization of the usual
recurrence for $R$-polynomials, using the notion of an 
inversion-minimal transposition on the interval $(u,v)$.  
The goal of the remainder of the paper is to discuss 
Bruhat interval polytopes for $G/P$.  In Section \ref{sec:flag} we provide
background on generalized partial flag varieties $G/P$, including
generalized Pl\"ucker coordinates, the Gelfand-Servanova stratification
of $G/P$, total positivity, and the moment map.  Finally in Section
\ref{sec:genBIP}, we show that each cell in Rietsch's 
cell decomposition of $(G/P)_{\geq 0}$ lies in a Gelfand-Serganova
stratum, and we use this result to prove that the face of 
a Bruhat interval polytope for $G/P$ is again a Bruhat interval polytope
for $G/P$.

\textsc{Acknowledgments:}
E.T. is grateful to Francesco Brenti 
for sending his Maple code for computing 
R-polynomials.   L.W. would like to thank Yuji Kodama for 
their joint work on the Toda lattice, which provided motivation for 
this project.
She is also grateful to Konni Rietsch and Robert Marsh for generously sharing
their ideas about 
the Gelfand-Serganova stratification.  Finally, we would like to thank
Fabrizio Caselli, Allen Knutson, Thomas Lam, Robert
Proctor, and Alex Woo for helpful
conversations.

This material is based upon work supported by the National Science Foundation Graduate Research Fellowship under Grant No. DGE 1106400,
and the National Science Foundation 
CAREER Award DMS-1049513. Any opinion, findings, and conclusions or recommendations expressed in this material are those of the authors(s) and do not necessarily reflect the views of the National Science Foundation.

\section{Background}\label{sec:background}

In this section we will quickly review some notation and background 
for posets and Coxeter groups.  We will also review some basic facts about
permutohedra, matroid polytopes, and Bruhat interval polytopes.
We will assume knowledge of the basic definitions
of Coxeter systems and Bruhat order; we refer the reader to
\cite{BB}
for details.  Note that throughout
this paper, Bruhat order will refer to the strong Bruhat order.

Let $P$ be a poset with order relation $<$.  We will use the symbol
$\lessdot$ to denote a covering relation in the poset: $u \lessdot
v$ means that $u < v$ and there is no $z$ such that $u < z < v$.
Additionally, if $u < v$ then $[u,v]$ denotes the {\it (closed) interval}
from $u$ to $v$; that is, $[u,v]=\{z \in P \ \vert \ u \leq z \leq
v \}$.  Similarly, $(u,v)$ denotes the {\it (open) interval}, that is,
$(u,v) = \{z\in P \ \vert \ u < z < v\}$.

The natural geometric object that one associates to a poset $P$ is
the geometric realization of its {\it order complex} (or {\it nerve}). The
order complex $\Delta(P)$ is defined to be the simplicial complex
whose vertices are the elements of $P$ and whose simplices are the
chains $x_0 < x_1 < \dots < x_k$ in $P$. Abusing notation,
we will also use the notation $\Delta(P)$ to denote the geometric 
realization of the order complex.

Let $(W,S)$ be a Coxeter group generated by a set of simple reflections
$S=\{ s_i \ \vert \ i \in I \}$.
We denote the set of all reflections by $T = \{wsw^{-1} \ \vert \ w\in W\}$.
Recall that a \emph{reduced word} for an element $w\in W$
is a minimal length expression for $w$ as a product of 
elements of $S$, and the \emph{length} $\ell(w)$ of $w$
is the length of a reduced word.
For $w\in W$, we let $D_R(w)= \{s\in S\ \vert \ ws \lessdot w\}$
be the \emph{right descent set} of $w$ and $D_L(w)= \{s\in S\ \vert \ sw \lessdot w\}$
the \emph{left descent set} of $w$. We also let $T_R(w)=\{t\in T\ \vert \ \ell(wt)<\ell(w) \}$ and $T_L(w)=\{t\in T\ \vert \ \ell(tw)<\ell(w) \}$ be the \emph{right associated reflections} and \emph{left associated reflections} of $w$, respectively.

The \emph{(strong) Bruhat order} on $W$ is defined by 
$u \leq v$ if some substring of some (equivalently, every)
reduced word for $v$ is a reduced word for $u$.
The Bruhat order on a Coxeter group is a graded poset, with rank
function given by length.

When $W$ is the symmetric group $S_n$, the reflections 
are the transpositions $T = \{(ij) \ \vert \ 1 \leq i<j \leq n\}$, the set of 
permutations which act on $\{1,\dots,n\}$ by swapping $i$ and $j$.
The simple reflections are the reflections of the form
$(ij)$ where $j=i+1$.  We also denote this simple reflection by $s_i$.
An \emph{inversion} of a permutation $z = (z(1),\dots,z(n))\in S_n$
is a pair $(ij)$ with $1 \leq i <j \leq n$ such that 
$z(i)>z(j)$.  It is well-known that $\ell(z)$ is equal to the number of 
inversions of the permutation $z$.

Note that we will often use the notation $(z_1,\dots, z_n)$ instead
of $(z(1),\dots, z(n))$.

We now review some facts about  permutohedra, matroid polytopes, and Bruhat interval polytopes.

\begin{definition}
The \emph{usual permutohedron} $\Perm_n$ in $\R^n$ is the convex
hull of the $n!$ points obtained by permuting the coordinates
of the vector $(1,2,\dots,n)$.
\end{definition}

Bruhat interval polytopes, as defined below, 
were introduced and studied by 
Kodama and the second author in \cite{KW3}, 
in connection with the full Kostant-Toda lattice on the 
flag variety.

\begin{definition}\label{def:BIP}
Let $u,v \in S_n$ such that $ u \leq v $ in (strong) Bruhat
order.  We identify each permutation $z\in S_n$ with
the corresponding vector $(z(1),\dots, z(n)) \in \R^n.$
Then the \emph{Bruhat interval polytope}
$\mathsf{Q}_{u,v}$ is defined as the convex hull of all
vectors $(z(1),\dots,z(n))$ for $z$ such that $u \leq z \leq v$.
\end{definition}

See Figure \ref{fig:Mpoly} for some examples of Bruhat interval
polytopes.

\begin{figure}[h]
\centering
\includegraphics[height=5.3cm]{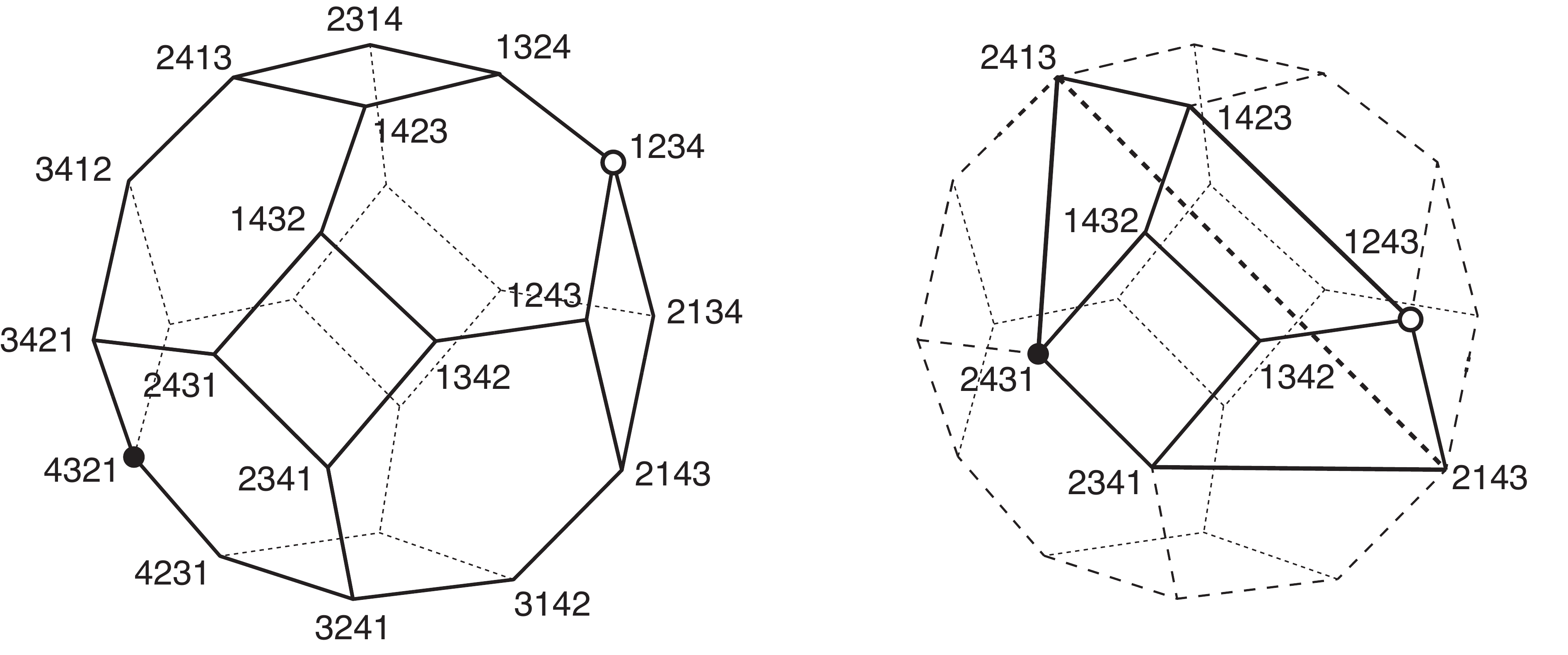}
\caption{The two polytopes are the permutohedron
 $\mathsf{Q}_{e,w_0}=\Perm_4$, and the 
Bruhat interval polytope $\mathsf{Q}_{u,v}$ with $v=(2,4,3,1)$ and $u=(1,2,4,3)$.
\label{fig:Mpoly}}
\end{figure}

We next explain how Bruhat interval polytopes are related to matroid polytopes,
generalized permutohedra, and flag matroid polytopes.

\begin{definition}
Let $\M$ be a nonempty
collection
of $k$-element subsets of $[n]$ such that:
if $I$ and $J$ are distinct members of $\M$ and
$i \in I \setminus J$, then there exists an element $j \in J \setminus I$
such that $(I \setminus \{i\}) \cup \{j\} \in \M$.
Then $\M$ is called the \emph{set of bases of a matroid} of \emph{rank $k$}
on the \emph{ground set} $[n]$;
or simply a \emph{matroid}.
\end{definition}

\begin{definition}
Given the set of bases $\M \subset {[n] \choose k}$
of a matroid, the \emph{matroid polytope}
$\Gamma_{\M}$ of $\M$ is the convex hull of the indicator vectors of the bases of $\M$:
\[
\Gamma_{\M} := \convex\{e_I \mid I \in \M\} \subset \R^n,
\]
where $e_I := \sum_{i \in I} e_i$, and $\{e_1, \dotsc, e_n\}$ is the standard basis of $\R^n$.
\end{definition}

Note that ``a matroid polytope" refers to the polytope of a specific matroid in its specific position in $\R^n$.

\begin{definition}
The \emph{flag variety} $\Fl_n$ is the variety of all flags
$$\Fl_n = \{V_{\bullet} = V_1 \subset V_2 \subset \dots \subset V_n = \R^n \ \vert \ \dim V_i = i\}$$
of vector subspaces of $\R^n$.
\end{definition}

\begin{definition}
The \emph{Grassmannian} $Gr_{k,n}$ is the variety of $k$-dimensional subspaces of $\R^n$
$$Gr_{k,n} = \{V \subset \R^n \ \vert \ \dim V = k\}.$$
\end{definition}

Note that there is a natural projection $\pi_k: Fl_n \to Gr_{k,n}$ taking 
$V_{\bullet} = V_1 \subset \dots \subset V_n$ to $V_k$.

Note also that any element $V \in Gr_{k,n}$ gives rise to a matroid
$\M(V)$ of rank $k$ on the ground set $[n]$.  
First represent $V$ as the row-span of a full rank $k \times n$ matrix $A$.
Given a $k$-element subset $I$ of $\{1,2,\dots,n\}$,
let $\Delta_I(A)$ denote the determinant of the $k \times k$ submatrix of $A$
located in columns $I$.  This is called a \emph{Pl\"ucker coordinate}.
Then $V$ gives rise to a matroid $\M(V)$ whose bases
are precisely the $k$-element subsets $I$ such that
$\Delta_I(A) \neq 0$.

One result of \cite[Section 6]{KW3} (see also \cite[Appendix]{KW3}) is the following.
See Section \ref{sec:flag} for the definition of $\mathcal{R}_{u,v;>0}$.

\begin{proposition}\label{prop:Mink-sum}
Choose $u \leq v \in S_n$.
Let $V_{\bullet}=V_1 \subset \dots \subset V_n$ be any element in the positive part of the Richardson
variety
$\mathcal{R}_{u,v;>0}$.  
Then the Bruhat interval polytope
$\mathsf{Q}_{u,v}$ is the Minkowski sum of $n-1$ matroid polytopes:
$$\mathsf{Q}_{u,v} = \sum_{k=1}^{n-1} \Gamma_{\M(V_k)}.$$
In fact each of the polytopes $\Gamma_{\M(V_k)}$ is a \emph{positroid polytope},
in the sense of \cite{ARW}, and 
$\mathsf{Q}_{u,v}$ is a \emph{generalized permutohedron}, in the sense of 
Postnikov \cite{Postnikov2}.
\end{proposition}

We can compute the bases $\M(V_k)$ 
from the permutations $u$ and $v$
as follows.
\begin{equation}\label{eq:projection}
\M(V_k) = \{I \in {[n] \choose k} \ \vert \ \text{ there exists }
z \in [u,v] \text{ such that }I=\{z^{-1}(n), z^{-1}(n-1),
\dots, z^{-1}(n-k+1)\}\}.
\end{equation}

Therefore we have the following.
\begin{proposition}\label{prop:Mink-sum2}
For any $u \leq v \in S_n$, the Bruhat interval polytope 
$\mathsf{Q}_{u,v}$ is the Minkowski sum of $n-1$ matroid polytopes
$$\mathsf{Q}_{u,v} = \sum_{k=1}^{n-1} \Gamma_{\M_k},$$ where
\begin{equation*}
\M_k = \{I \in {[n] \choose k} \ \vert \ \text{ there exists }
z \in [u,v] \text{ such that }I=\{z^{-1}(n), z^{-1}(n-1),
\dots, z^{-1}(n-k+1)\}\}.
\end{equation*}
\end{proposition}

\emph{Positroid polytopes} are a particularly nice class of matroid polytopes coming
from positively oriented matroids.  A \emph{generalized 
permutohedron} is a polytope which is
obtained by moving the vertices
of the usual permutohedron in such a way that directions of
edges are preserved, but some edges (and higher dimensional faces)
may degenerate.  See \cite{ARW} and \cite{Postnikov2} for more details
on positroid polytopes and generalized permutohedra.

There is a generalization of matroid called \emph{flag matroid}, due to 
Gelfand and Serganova \cite{GelfSerg}, \cite[Section 1.7]{CoxeterMatroids},
and a corresponding 
notion of \emph{flag matroid polytope}.  A convex polytope $\Delta$ in the real vector space $\R^n$
is called a (type $A_{n-1}$) \emph{flag matroid polytope} if the edges of $\Delta$ are parallel to the roots 
of type $A_{n-1}$ and there exists a point equidistant from all of its vertices.

The following result follows easily from Proposition \ref{prop:Mink-sum}.

\begin{proposition}\label{prop:flagmatroid}
Choose $u \leq v \in S_n$.
Then the Bruhat interval polytope
$\mathsf{Q}_{u,v}$ is a flag matroid polytope.
\end{proposition}

\begin{proof}
Let $V_{\bullet}=V_1 \subset \dots \subset V_n$ be any element in the positive part of the Richardson
variety
$\mathcal{R}_{u,v;>0}$.  
By Proposition \ref{prop:Mink-sum},
$\mathsf{Q}_{u,v} = \sum_{k=1}^{n-1} \Gamma_{\M(V_k)}.$
Then \cite[Theorem 1.7.3]{CoxeterMatroids}
implies that
the collection of matroids $\M_{\bullet} = \{\M(V_1),\dots, \M(V_{n-1})\}$
forms a flag matroid.
By \cite[Theorem 1.13.5]{CoxeterMatroids},
it follows that the flag matroid polytope associated to 
$\M_{\bullet}$ is the Minkowski sum of the matroid polytopes
$\Gamma_{\M(V_1)},\dots, \Gamma_{\M(V_{n-1})}$.
Therefore $\mathsf{Q}_{u,v}$ is a flag matroid polytope.
\end{proof}

We can use Proposition \ref{prop:flagmatroid} to prove the following useful result.

\begin{proposition} \label{minmax}
Let $\PPP_{u,v}$ be a Bruhat interval polytope. Consider a face $F$ of $\PPP_{u,v}$. Let $\mathcal{N}$ be the set of permutations which label vertices of $F$. Then $\mathcal{N}$ contains an element $x$ and an element $y$ such that 
\[
x \leq z \leq y \quad \forall z \in \mathcal{N}.
\]
\end{proposition}

\begin{proof}
By Proposition \ref{prop:flagmatroid}, $\mathsf{Q}_{u,v}$ is a flag matroid polytope.
It follows from the definition that every face of a flag matroid polytope is again a flag matroid polytope,
and therefore the face $F$ is a flag matroid polytope.
By \cite[Section 6.1.3]{CoxeterMatroids}, every flag matroid is a Coxeter matroid, and hence
the permutations $\mathcal{N}$ labeling the vertices of $F$ are the elements of a Coxeter matroid
(for $S_n$, with parabolic subgroup the trivial group).  But now by the Maximality Property for 
Coxeter matroids \cite[Section 6.1.1]{CoxeterMatroids}, 
$\mathcal{N}$ must contain a minimal element $x$ such that $x \leq z$ for all $z\in \mathcal{N}$,
and $\mathcal{N}$ must contain a maximal element $y$ such that $y \geq z$ for all $z\in \mathcal{N}$.
\end{proof}

\section{The generalized lifting property for the symmetric group}\label{sec:lifting}

The main result of this section is Theorem \ref{GeneralizedLiftingProperty},
which is a generalization
(for the symmetric group) of the classical lifting property
for Coxeter groups.  This result will be a main tool for
proving 
that every face of a Bruhat interval polytope is a Bruhat interval polytope.

We start by recalling the usual lifting property.

\begin{proposition}[Lifting property] \label{lifting1} Suppose $u<v$ 
and $s \in D_R(v) \setminus D_R(u)$. Then $u \leq vs \lessdot v$ and $u \lessdot us \leq v$.
\end{proposition}

\begin{definition}\label{def:inversion-minimal}
 Let $u,v \in S_n$. A transposition $(i k)$ is \emph{inversion-minimal on $(u,v)$} if the interval $[i,k]$ is the 
minimal interval (with respect to inclusion)  which has the property
\[
v_i > v_k, \quad u_i < u_k.
\]
\end{definition}
   
\begin{theorem}[Generalized lifting property]\label{GeneralizedLiftingProperty}
Suppose $u<v$ in $S_n$. 
Choose a transposition $(ij)$ which is inversion-minimal on $(u,v)$. Then $u \leq v(i j)
\lessdot v$ and $u \lessdot u (i j) \leq v$.
\end{theorem}

We note that there are pairs $u<v$ where $D_R(v) \setminus D_R(u)$ is empty,
and hence one cannot apply the Lifting property.  In contrast, Lemma \ref{iklemma}
below shows that 
for any pair $u<v$ in $S_n$, there exists an inversion-minimal transposition $(i j)$.
Hence it is always possible to apply the Generalized lifting property.

\begin{lemma}  \label{iklemma}
Let $(W,S)$ be a Coxeter group. Take $u,v \in W$ distinct. If $\ell(v) \geq \ell(u)$ then there exists a reflection $t \in T$ such that 
\[
v > vt, \quad u < ut.
\]
\end{lemma}

\begin{proof}
Recall that $T_R(w)= \{t \in T\ \vert \ wt < w\}$.  
The lemma will follow if we show that $T_R(v) \not \subset T_R(u)$. Assume by contradiction that $T_R(v) \subset T_R(u)$. By \cite[Corollary 1.4.5]{BB}, for any $x \in W$, $|T_R(x)|=\ell(x)$. Since $\ell(v) \leq \ell(u)$, we must 
have $T_R(v)=T_R(u)$.  By  \cite[Chapter 1 Exercise 11]{BB}, this contradicts $v \neq u$.
\end{proof}

Lemma \ref{iklemma} directly implies the following corollary.

\begin{corollary}
Let $v,u \in S_n$ be two distinct permutations. If $\ell(v) \geq \ell(u)$ then there exists an inversion-minimal transposition on $(u,v)$.
\end{corollary}

In preparation for the proof of 
Theorem \ref{GeneralizedLiftingProperty},
it will be convenient to make the following definition.

\begin{definition}  A \emph{pattern} of length $n$ is an equivalence class of sequences $x_1 x_2 \cdots x_n$ of distinct integers. Two such sequences $x_1 x_2 \cdots x_n$, $y_1 y_2 \cdots y_n$ are in the same equivalence class (``have the same pattern'') if 
\[
x_i > x_j \iff y_i > y_j \hspace{.2cm}\text{ for all $i$, $j$ such that } 1\leq i,j \leq n.
\]
 Denote by $\Patt_n$ the set of patterns of length $n$.
\end{definition}

There is a canonical representative for each pattern $x \in \Patt_n$ obtained by replacing each $x_i$ with
\[
\bar{x}_i:=\#\{j \in [n]: x_j \leq x_i\}.
\]
For example, the canonical representative of $523$ is $312$.

\begin{definition}
Let $x,y \in \Patt_n$ for some $n$. Call $(x,y)$ an Inversion-Inversion pair if the following condition holds:
\[
\text{for all } i < j , \quad x_i>x_j \implies y_i>y_j.
\]
\end{definition}

Notice that this statement is independent of the choice of representatives. 

It is easy to see that if $(x,y)$ is an Inversion-Inversion pair, then so is $(x_1 \cdots \hat{x}_k \cdots x_n, y_1 \cdots \hat{y}_k \cdots y_n)$ for any $k$.

In preparation for the proof of Theorem \ref{GeneralizedLiftingProperty},
we first state and prove Lemmas \ref{minimalEqualsII},
\ref{INI}, and 
\ref{lem:dot-bijection}.

\begin{lemma} \label{minimalEqualsII}

Let $u,v \in S_n$. The following are equivalent:
\begin{enumerate}
\item[(i).] The transposition $(i k)$ is inversion-minimal on $(u,v)$
\item[(ii).] The patterns $x=x_i \dots x_k := v_i \cdots v_k$ and 
$y=y_i \dots y_k:=u_k u_{i+1} u_{i} \cdots u_{k-2} u_{k-1} u_i$ form an Inversion-Inversion pair $(x,y)$ with $\bar{x}_k=\bar{x}_i+1$ and $\bar{y}_k=\bar{y}_i+1$.
\end{enumerate}
\end{lemma}

\begin{proof}
First note that (ii) obviously implies (i).  We now prove that 
(i) implies (ii). Assume transposition $(i k)$ is inversion-minimal.
 We show that the following two cases cannot hold:\\
\textbf{Case 1}: there is some $i<j<k$ with $v_j \in [v_k,v_i]$.

Looking at intervals $[i,j]$ and $[j,k]$, we have $v_i>v_j$ and $v_j>v_k$. By minimality of $[i,k]$, this implies that $u_i > u_j$ and $u_j > u_k$, contradicting $u_i < u_k$. \\
\textbf{Case 2}: there is some $i<j<k$ with $u_j  \in [u_i,u_k]$.

Looking at intervals $[i,j]$ and $[j,k]$ again, we have $u_i<u_j$ and $u_j<u_k$. By minimality of $[i,k]$, this implies that $v_i < v_j$ and $v_j < v_k$, contradicting $v_i > v_k$. \\
(ii) $\implies$ (i). Since $\bar{x}_k>\bar{x}_1$ and $\bar{y}_k>\bar{y}_1$, we see that $v_i>v_k$ and $u_i<u_k$. Assume by contradiction that $[p,q]$ is a strict subset of $[i,k]$ such that $v_p>v_q$ and $u_p<u_q$. Since $\bar{x}_k=\bar{x}_i+1$ and $\bar{y}_k=\bar{y}_i+1$, for any $j \in (i,k)$, 
\[
x_j > x_k \iff x_j > x_i
\] 
\[
y_j > y_k \iff y_j > y_i. 
\]
Equivalently, for any $j \in (i,k)$,
\[
v_j > v_k \iff v_j > v_i
\] 
\[
u_j > u_k \iff u_j > u_i. 
\]
If $\{p,q\} \cap \{i,k\} = \emptyset$, then we clearly obtain a contradiction.
If $p=i$, then 
\[
v_i>v_q, u_i<u_q \implies x_i>x_q, y_k<y_q \implies x_i>x_q, y_i<y_q,
\]
which is a contradiction. A similar argument shows that $q=k$ leads to a contradiction.
\end{proof}

Lemma \ref{minimalEqualsII}
implies the following result.
\begin{corollary}\label{cor:cover}
Let $u,v \in S_n$ and let $(ik)$ be inversion-minimal on $(u,v)$. Then
\[
v (ik) \lessdot v \quad \text{ and } \quad u \lessdot u(ik).
\]

\end{corollary}

\begin{center}
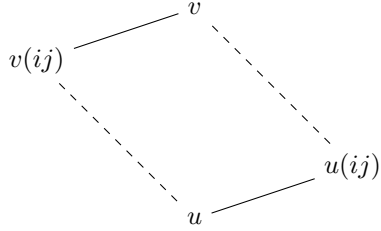

\begin{tikzpicture}[scale=.7] 
  \node (v) at (0,2) {$v$};
  \node (vij) at (-3,1) {$v(ij)$};
  \node (uij) at (3,-1) {$u(ij)$};
  \node (u) at (0,-2) {$u$};
  \draw  (u) - - (uij);
  \draw (v) -- (vij); 
 \draw [dashed] (vij) -- (u);
 \draw [dashed] (uij) -- (v);
\end{tikzpicture}
\captionof{figure}{Generalized lifting property}
   \label{liftv2}
   \end{center}

\begin{lemma} \label{INI}
Let $x,y \in \Patt_n$ with $\bar{x}_n=\bar{x}_1+1$ and $\bar{y}_n=\bar{y}_1+1$. If
$(x,y)$ is an Inversion-Inversion pair, then $\bar{x}_1=\bar{y}_1$.
\end{lemma}

\begin{proof}
Define
\[
I_{i,j}(x):=\begin{cases} 1\quad \text{ if }x_i>x_j,\\
                                               
        0 \quad\text{ if }x_i<x_j.
\end{cases}
\]
This function is well-defined on patterns.
Let
\[
f(x,y):= \sum_{1 \leq i < j \leq n} I_{i,j}(x)(1-I_{i,j}(y)).
\]
With this notation, $(x,y)$ is an Inversion-Inversion pair if and only if $f(x,y)=0$.

Note that 
\[
f(x,y)=\ell(x)-\sum_{1 \leq i <j \leq n} I_{i,j}(x)I_{i,j}(y).
\]

The pairs
\[
(a):(x_1 \cdots x_{n-1},y_1\cdots y_{n-1})
\]
 and 
\[ 
 (b):(x_2\cdots x_n,y_2 \cdots y_n)
\] 
  are Inversion-Inversion pairs.
The conditions on $\bar{x}_1,\bar{x}_n$ imply that 
\begin{align} \label{x1xn}
I_{1,j}(x)=1-I_{j,n}(x), \quad \forall 1<j<n
\end{align}
and similarly for $y$. Since $f(x_1 \cdots x_{n-1},y_1\cdots y_{n-1})=0$, and using $I_{1,n}(x)=0$,
\begin{align} \label{a}
(a): \ell(x)-\sum_{1<i<n} I_{i,n}(x)= \sum_{1\leq i<j<n} I_{i,j}(x)I_{i,j}(y)
\end{align}
Applying condition (\ref{x1xn}) to (\ref{a}), and simplifying, we get
\begin{align} \label{afinal}
(a):\ell(x)-(n-2)+\sum_{1<j<n} I_{1,j}(x)=\sum_{1<j<n} I_{1,j}(x)I_{1,j}(y)+\sum_{1<i <j<n} I_{i,j}(x)I_{i,j}(y)
\end{align}
Similarly, since $f(x_2\cdots x_n,y_2 \cdots y_n)=0$ and $I_{1,n}(x)=0$,
\begin{align} \label{b}
(b):\ell(x)-\sum_{1<j<n} I_{1,j}(x)= \sum_{1<i \leq j  \leq n} I_{i,j}(x)I_{i,j}(y)
\end{align}
Using condition \eqref{x1xn} with $x$ replaced with $y$, equation (\ref{b}) reduces to 
\[
\ell(x)-\sum_{1<j<n} I_{1,j}(x)=\sum_{1<i<n}I_{i,n}(x)I_{i,n}(y)+ \sum_{1<i < j  < n} I_{i,j}(x)I_{i,j}(y)
\]
\[
=\sum_{1<j<n}(1-I_{1,j}(x))(1-I_{1,j}(y))+ \sum_{1<i < j  < n} I_{i,j}(x)I_{i,j}(y)
\]
\begin{align} \label{bfinal}
=(n-2)-\sum_{1<j<n}(I_{1,j}(x)+I_{1,j}(y))+\sum_{1<j<n} I_{1,j}(x)I_{1,j}(y)+\sum_{1<i < j  < n} I_{i,j}(x)I_{i,j}(y).
\end{align}
Comparing (\ref{afinal})  and (\ref{bfinal}) we see that
\[
 \sum_{1<j<n}I_{1,j}(x)=  \sum_{1<j<n}I_{1,j}(y)
\]
which can only happen if $\bar{x}_1=\bar{y}_1$.
\end{proof}

\begin{lemma}\label{lem:dot-bijection}
Suppose that $(i k)$ is inversion-minimal on $(u,v)$.  Then
for every $i<j<k$, we have
$$u_j > u_i \iff u_j > u_k \iff v_j > v_k \iff v_j > v_i.$$
\end{lemma}

\begin{proof}
By Lemma \ref{minimalEqualsII}, the patterns $x=v_i \cdots v_k$ and $y=u_k u_{i+1} \cdots u_{k-1} u_i$ form an Inversion-Inversion pair $(x,y)$ with $\bar{x}_k=\bar{x}_1+1$ and $\bar{y}_k=\bar{y}_1+1$. By Lemma \ref{INI}, $\bar{x}_1=\bar{y}_1$.
It follows that
\[
\#\{j : i < j < k, v_j > v_k\}=\#\{j : i < j < k, u_j > u_i\}.
\]
We also see that $\#\{j : i < j < k, v_j > v_k\}=\#\{j : i < j < k, v_j > v_i\}$ and $\#\{j : i < j < k, u_j > u_i\}=\#\{j : i < j < k, u_j > u_k\}$.
By minimality of $[i,k]$, for every $i < j <k$, 
\[
v_j > v_k \implies u_j > u_k.
\]
Consequently, for every $i < j < k$,
\begin{align} \label{iffs}
u_j > u_i \iff u_j > u_k \iff v_j > v_k \iff v_j > v_i.
\end{align}
\end{proof}

Finally we are ready to prove Theorem \ref{GeneralizedLiftingProperty}.

\begin{proof} [Proof of Theorem \ref{GeneralizedLiftingProperty}].
Choose $u<v$ in $S_n$, and a transposition $(i j)$ which is inversion-minimal on $(u,v)$.
By Corollary \ref{cor:cover}, to prove Theorem \ref{GeneralizedLiftingProperty}, it suffices
to show that $u \leq v(ij)$ and $u(ij) \leq v$.

We use induction on $k=j-i$. The base case $k=1$ holds by the lifting property (Proposition \ref{lifting1}). 

Now consider $k>1$. Since $(i j)$ is inversion-minimal on $(u,v)$, we have $v_i>v_j$ and $u_i<u_j$.

\textbf{Case 1}: Suppose that (a) $v_i>v_{i+1}$ and $u_i>u_{i+1}$, or (b) $v_i<v_{i+1}$ and $u_i<u_{i+1}$.

We have (a) $u \gtrdot us_i$ and $v \gtrdot vs_i$ or (b) $u \lessdot us_i$ and $v \lessdot vs_i$. Clearly $((i+1) j)$ is inversion-minimal on $(vs_i,us_i)$, and since $u<v$, we have $us_i < vs_i$. By induction,
\[
us_i \leq vs_i ((i+1)j) \text{ and } us_i ((i+1)j) \leq vs_i.
\]
Notice that $s_i((i+1)j)s_i=(ij)=t$. 

In case (a), we 
claim that $s_i \notin D_R(us_i) \cup D_R(v s_i ((i+1)j))$.
To see this, note first that 
$v_{i+1}<v_j$; otherwise we'd have
$v_{i+1}>v_j$ and also  
$u_{i+1}<u_j$, which would contradict our 
assumption that the interval $[i,j]$ is inversion-minimal on $(u,v)$.
Therefore $s_i \notin D_R(v s_i ((i+1)j))$, and the claim follows.
But now the claim together with 
$us_i \leq vs_i ((i+1)j)$ implies that 
$us_i^2 \leq vs_i ((i+1)j) s_i$ and hence 
$u \leq vt$.

In case (b), we claim that 
$s_i \in D_R(us_i ((i+1)j)) \cap D_R(vs_i)$.
To see this, note first that 
$u_{i+1}>u_j$; otherwise we'd have
$u_{i+1}<u_j$ and also $v_{i+1}>v_j$, which would contradict our
assumption that transposition $(i j)$ is inversion-minimal on $(u,v)$.
Therefore $s_i \in D_R(u s_i ((i+1)j))$, and the claim follows.
But now the claim together with 
$us_i ((i+1)j) \leq vs_i$ implies that 
$u s_i ((i+1)j) s_i \leq vs_i^2$, and hence
$u t \leq v$.

\textbf{Case 2}: Suppose that 
$v_{j-1}>v_{j}$ and $u_{j-1}>u_j$, or $v_{j-1}<v_j$ and $u_{j-1}<u_j$.

This case is analogous to Case 1.

\textbf{Case 3}: Suppose that neither of the above two cases holds.

Since $(i j)$ is inversion-minimal on $(u,v)$, we must have $v_i<v_{i+1}$ and $v_{j-1}<v_j$. Since $v_i>v_j$, there exists some $m_1 \in (i,j-1)$ such that $v_{m_1}>v_{m_1+1}$. By minimality, $u_{m_1}>u_{m_1+1}$. By Lemma \ref{minimalIntermediate}, $(i j)$ is inversion-minimal on $(vs_{m_1},us_{m_1})$. If $us_{m_1}$ and $vs_{m_1}$ do not satisfy the conditions of Cases 1 or 2, then we may find $m_2 \in (i,j-1)$ and then $(i j)$ is inversion-minimal on $(vs_{m_1} s_{m_2},us_{m_1} s_{m_2})$. Such a sequence $m_1,m_2,\ldots$ clearly terminates. Assume that it terminates at $k$, so that $(i j)$ is inversion-minimal on $(vs_{m_1} s_{m_2} \cdots s_{m_k},u s_{m_1}s_{m_2} \cdots s_{m_k})$ and the hypotheses of Case 1 or 2 are satisfied for $vs_{m_1} s_{m_2} \cdots s_{m_k}$ and $u s_{m_1}s_{m_2} \cdots s_{m_k}$. Set $\Pi_k:=s_{m_1}s_{m_2} \cdots s_{m_k}$. We then have

\[
u\Pi_k \lessdot u\Pi_k t, \quad v\Pi_k t \lessdot v \Pi_k, \quad u \Pi_k t \leq v \Pi_k , \quad u \Pi_k \leq v \Pi_k t.
\]
We show now that for $1 \leq p \leq k$, if 
\[
u\Pi_p \lessdot u\Pi_p t, \quad v\Pi_p t \lessdot v \Pi_p, \quad u \Pi_p t \leq v \Pi_p , \quad u \Pi_p \leq v \Pi_p t
\]
then
\[
u\Pi_{p-1} \lessdot u\Pi_{p-1} t, \quad v\Pi_{p-1} t \lessdot v \Pi_{p-1}, \quad u \Pi_{p-1} t \leq v \Pi_{p-1} , \quad u \Pi_{p-1} \leq v \Pi_{p-1} t.
\]

Note that for any $m$, $ts_m=s_mt$. Therefore $u\Pi_{p}t=u\Pi_{p-1}t s_{m_p}$ and $v\Pi_{p}t=v\Pi_{p-1}t s_{m_p}$. This implies that $u\Pi_{p}t=u\Pi_{p-1}t s_{m_p} \gtrdot u\Pi_{p-1}t$ and $v\Pi_{p}t=v\Pi_{p-1}t s_{m_p} \lessdot v\Pi_{p-1}t$.
\end{proof}

\begin{center}
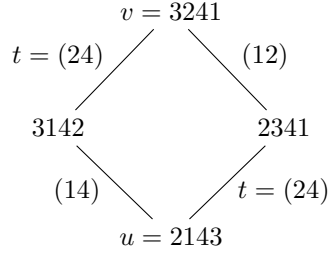

\begin{tikzpicture}[scale=.5] 
\node (u) at (0,-3) {{$u=2143$}};
  \node (a) at (-3,-0.1) {{$3142$}};
  \node (b) at (3,-0.1) {{$2341$}};
  \node (v) at (0,3) {{$v=3241$}};
  \draw  (u) -- (a);
  \draw (u) -- (b); 
  \draw (a) -- (v);
  \draw (b) -- (v);
  \node (x1) at (3,-1.8) {{$t=(24)$}};
  \node (x2) at (-2.5,-1.8) {{$(14)$}};
  \node (x3) at (-3,1.8) {{$t=(24)$}};
  \node (x4) at (2.5,1.8) {{$(12)$}};
\end{tikzpicture}
\captionof{figure}{Example of Theorem \ref{GeneralizedLiftingProperty}}
   \label{exliftv2}
    \end{center}

\begin{example}
The following example shows that the converse to Theorem  \ref{GeneralizedLiftingProperty} does not hold: it is not necessarily the case that if the Bruhat relations \[v(ik) \lessdot v \quad
u \lessdot u(ik) \quad
u \leq v(ik) \quad
u(ik) \leq v\]
hold, then $(i k)$ is inversion-minimal on $(u,v)$. 
Take $v=4312$, $u=1243$ and $(ik)=(24)$. Then 
\[v(ik) \lessdot v \quad
u \lessdot u(ik) \quad
u \leq v(ik) \quad
u(ik) \leq v\]
but also $v_2>v_3$ and $u_2<u_3$.
\end{example}

As a corollary of Generalized lifting, we have the following result, which says that in an interval of the symmetric group we may find a maximal chain such that each transposition connecting two consecutive elements of the chain is a transposition that comes from the atoms, and similarly, for the coatoms.
\begin{corollary} \label{atomCoatomChains}
Let $[u,v]= \subset S_n$ and let $\underline{T}(v):=\{t  \in T : v \gtrdot vt \geq u\}$ and $\overline{T}(u):=\{t  \in T : u \lessdot ut \leq v\}$. There exist maximal chains $\mathcal{C}_v:u=x_{(0)} \lessdot x_{(1)} \lessdot x_{(2)} \lessdot \ldots \lessdot x_{(l)}=v$ and $\mathcal{C}_u:u=y_{(0)} \lessdot y_{(1)} \lessdot y_{(2)} \lessdot \ldots \lessdot y_{(l)}=v$ in $I$ such that $x_{(i)}^{-1}x_{(i+1)} \in \underline{T}(v)$ and $y_{(i)}^{-1}y_{(i+1)}\in \overline{T}(u)$ for each $i$.
\end{corollary}

\begin{proof}
By the Generalized lifting property,  
there exists a transposition $t=(ij)$ such that $u \leq vt \lessdot v$ and $u \lessdot ut \leq v$. 
But now since $u \lessdot ut \leq v$, 
we can apply the Generalized lifting property to the pair $ut \leq v$,
and inductively construct the maximal chain $\mathcal{C}_v$.  The 
construction of $\mathcal{C}_u$ is analogous.
\end{proof}

We plan to study the Generalized lifting property for other Coxeter groups in a separate paper.

\section{Results on Bruhat interval polytopes}\label{sec:faces}

In this section we give some results on Bruhat interval polytopes.  
We show that the face of a Bruhat interval polytope is a Bruhat
interval polytope; we give a dimension formula;  we give 
an inequality description; and we give a criterion for when
one Bruhat interval polytope is a face of another.

\subsection{Faces of Bruhat interval polytopes are Bruhat interval polytopes}
The main result of this section is the following.

\begin{theorem} \label{bipFaceisBIP}
Every face of a Bruhat interval polytope is itself a Bruhat interval polytope.
\end{theorem}

Our proof of this result uses the following theorem.  It was 
first proved for the symmetric group by Edelman \cite{Edelman},
then generalized to classical types by Proctor \cite{Proctor}, and then 
proved for arbitrary Coxeter groups by Bjorner and Wachs \cite{BW}.
\begin{theorem}\cite{BW}\label{th:BW}
Let $(W,S)$ be a Coxeter group. Then for any $u \leq v$ in $W$,
the order complex $\Delta(u,v)$ of the 
interval $(u,v)$ is PL-homeomorphic to a sphere $\mathbb{S}^{\ell(u,v)-2}$.
In particular, the Bruhat order is \emph{thin}, that is,
every rank $2$ interval is a \emph{diamond}.  In other words,
whenever $u \leq v$ with $\ell(v)-\ell(u)=2$, 
there are precisely two elements
$z_{(1)}, z_{(2)}$ such that $u < z_{(i)} < v$.
\end{theorem}

We will identify a linear functional $\omega$ with 
a vector $(\omega_1,\ldots,\omega_n) \in \R^n$,
where $\omega:\R^n \rightarrow \R$ 
is defined by $\omega(e_i)=\omega_i$ (and extended linearly).

\begin{proposition}\label{prop:allchains}
Choose $u \leq v$ in $S_n$, and 
let $\omega:\R^n \to \R$ be a linear functional which is constant
on a maximal chain $\mathcal{C}$ from $u$ to $v$.  Then $\omega$ 
is constant on all permutations $z$ where $u \leq z \leq v$.
\end{proposition}

\begin{proof}
We will use the topology of $\Delta(u,v)$  to prove that $\omega$ is
constant on any maximal chain from $u$ to $v$.
If $\ell(v)-\ell(u)=1$, there is nothing to prove.  
If $\ell(v)-\ell(u)=2$, then the interval $[u,v]$ is a diamond.
By the Generalized lifting property
(Theorem \ref{GeneralizedLiftingProperty}),
there exists a transposition $t=(ij)$ such that $u \lessdot vt \lessdot v$ and $u \lessdot ut \lessdot v$.
Without loss of generality, $\mathcal{C}$ is the chain $u \lessdot vt \lessdot v$.  But then since $\omega(vt) = \omega(v)$,
we must have $\omega_i = \omega_j$.  It follows that $\omega(ut) =\omega(u)$,
and hence $\omega$ is constant on both maximal chains from $u$ to $v$.

If $\ell(v)-\ell(u) \geq 3$, then the order complex
$\Delta(u,v)$ is a PL sphere of dimension at least $1$, and hence
it is connected in codimension one.  
Therefore 
we can find a path of maximal chains $\mathcal{C} = \mathcal{C}_0, 
\mathcal{C}_1, \dots, \mathcal{C}_N$ in $(u,v)$ starting with 
$\mathcal{C}$, which contains \emph{all} maximal chains of $(u,v)$
(possibly some occur more than once),
and which has the property that for each adjacent pair
$\mathcal{C}_i$ and $\mathcal{C}_{i+1}$, the two chains 
differ in precisely one element.  
Since the Bruhat order is thin, 
$\mathcal{C}_i$ must contain three consecutive 
elements $a \lessdot z_{(1)} \lessdot b$, 
and $\mathcal{C}_{i+1}$ is obtained from $\mathcal{C}_i$
by replacing $z_{(1)}$ by $z_{(2)}$, the unique element other than $z_{(1)}$
in the interval $(a,b)$.
Suppose by induction that
$\omega$ is constant on $\mathcal{C}_0$,
$\mathcal{C}_1,\dots, \mathcal{C}_i$.  
Since 
$\omega(a)=\omega(z_{(1)}) = \omega(b)$ and $\ell(b)-\ell(a) = 2$,
we have observed in the previous paragraph
that $\omega$ must be constant on $[a,b]$.
Therefore $\omega$ attains the same value on $z_{(2)}$
and hence on all of $\mathcal{C}_{i+1}$.
\end{proof}

\begin{corollary} \label{max}
If a linear functional $\omega:\R^n \to \R$, when restricted to 
$[u,v]$, attains its maximum value on $u$ and $v$, then it is constant on 
$[u,v]$.
\end{corollary}

\begin{proof}
By Proposition \ref{prop:allchains}, it suffices
to show that there is a maximal chain
$\mathcal{C}_0 = \{u = z_{(0)} \lessdot z_{(1)} \lessdot \dots \lessdot z_{(\ell)} = v\}$
on which $\omega$ is constant.
By the Generalized lifting property,  
there exists a transposition $t=(ij)$ such that $u \leq vt \lessdot v$ and $u \lessdot ut \leq v$. Since $u \lessdot ut$ and $vt \lessdot v$,
we have $u_i < u_j$ and $v_i > v_j$.  Since $\omega(ut) \leq \omega(u)$,
it follows that $\omega_i \leq \omega_j$.  Similarly, 
$\omega(vt) \leq \omega(v)$ implies that $\omega_i \geq \omega_j$. Therefore $\omega_i=\omega_j$, and hence $\omega(ut)=w(u)=w(v)$. 
But now since $u \lessdot ut \leq v$, with $\omega(ut) = \omega(v)$,
we can apply the Generalized lifting property to the pair $ut \leq v$,
and inductively construct the desired maximal chain.
\end{proof}

We now prove the main result of this section.

\begin{proof}[Proof of Theorem \ref{bipFaceisBIP}].
Consider a face $F$ of a Bruhat interval polytope 
$\mathsf{Q}_{x,y}$ for $x,y\in S_n$.
Then there is a linear functional $\omega:\R^n \to \R$ which 
attains its maximum value $M$ precisely on the face $F$. 
By Proposition \ref{minmax}, there exist vertices 
$u,v \in F$ such that $u \leq z \leq v$ for each vertex $z\in F$. 
We want to show that $F = \mathsf{Q}_{u,v}$.  To complete
the proof, it suffices to show that every permutation $z$ such that $u \leq z \leq v$
lies in $F$, in other words, $\omega(z)=M$. But now since $\omega$ attains
its maximum value on $[u,v]$ on the permutations $u$ and $v$,
Corollary \ref{max} implies that $\omega$ is constant on $[u,v]$.
Therefore the vertices of $F$ are precisely the permutations in $[u,v]$.
\end{proof}

\subsection{The dimension of Bruhat interval polytopes}

In this section we will give a dimension formula for Bruhat interval polytopes.  We will
then use it to determine which Richardson varieties in $\Fl_n$ are toric varieties,
with respect to the usual torus action on $\Fl_n$.
Recall that a Richardson variety $\mathcal{R}_{u,v}$ is the intersection of opposite
Schubert (sometimes called Bruhat) cells; 
see Section \ref{sec:preliminaries} for background on Richardson varieties.

\begin{definition}
Let $u \leq v$ be permutations in $S_n$, and let $\mathcal{C}:u=x_{(0)} \lessdot x_{(1)} \lessdot x_{(2)} \lessdot \ldots \lessdot x_{(l)}=v$ be any maximal chain from $u$ to $v$.   Define a labeled graph $G^{\mathcal{C}}$ on $[n]$ having an edge between vertices $a$ and $b$ if and only if 
the transposition $(a b)$ equals
$x_{(i)}^{-1} x_{(i+1)}$ for some
$0 \leq i \leq l-1$. Define $B_{\mathcal{C}}=\{B^1,B^2,\ldots,B^r\}$ 
to be the partition of $[n]=\{1,2,\dots,n\}$ whose
blocks $B^j$ are the connected components of $G^{\mathcal{C}}$. Let $\#B_{\mathcal{C}}$ denote $r$, the number of blocks in the partition.
\end{definition}

We will show in Corollary \ref{B} that the partition $B_C$
is independent of $C$; and so we will denote this partition by $B_{u,v}$.

\begin{theorem}\label{th:dim}
The dimension $\dim \PPP_{u,v}$ of the Bruhat interval polytope $\PPP_{u,v}$ is 
\[
 \dim \PPP_{u,v} = n-\#B_{u,v}.
\]
The equations defining the affine span of $\PPP_{u,v}$ are
\begin{align} \label{affineSpaces}
\sum_{i \in B^j} x_i = \sum_{i \in B^j} u_i(=\sum_{i \in B^j} v_i), \quad j=1,2,\ldots,\#B_{u,v}.
\end{align}
\end{theorem}

Before proving Theorem \ref{th:dim}, we need to show that 
$B_{u,v}$ is well-defined.
Given a subset $A \subset [n]$,
let $e_A$ denote the $0-1$ vector in $\R^n$ with a $1$ in position
$a$ if and only if $a\in A$.

\begin{lemma} \label{wConst}
Let $\mathcal{C}$ be a maximal chain in $[u,v] \subset S_n$. Let 
$B_{\mathcal{C}}=\{B^1,\dots, B^r\}$ be the associated partition of 
$[n]$.
Then a linear functional 
$\omega:\R^n \rightarrow \R$ is constant on the interval 
$[u,v]$ 
if and only if 
\[
\omega=\sum_{j=1}^{r} c_j e_{B^j}
\]
for some coefficients $c_j$.
\end{lemma}
\begin{proof}
Using the definition of the partition $B_{\mathcal{C}}$, it is immediate that $\omega$ is constant on the chain $\mathcal{C}$ if and only if it 
has the form $\sum_{j=1}^{r} c_j e_{B^j}$.
The lemma now follows from Proposition \ref{prop:allchains}.
\end{proof}

\begin{corollary} \label{B}
The partition $B_{\mathcal{C}}$ is independent of the choice of $\mathcal{C}$.
\end{corollary}

\begin{proof}
Let $B_\mathcal{C}=\{B_\mathcal{C}^1,B_\mathcal{C}^2,\ldots,B_\mathcal{C}^r\}$. 
Take $\omega=e_{B_\mathcal{C}^j}$. By Lemma \ref{wConst}, $\omega$ is constant on $[u,v]$ and therefore on any other chain $\mathcal{C}'$. Consequently, there exist some elements $B_{\mathcal{C}'}^{j_1},\ldots,B_{\mathcal{C}'}^{j_k}$ of $B_{\mathcal{C}'}$ such that 
\[
B_\mathcal{C}^j=B_{\mathcal{C}'}^{j_1}\sqcup \ldots \sqcup B_{\mathcal{C}'}^{j_k}
\] 
It follows that $B_\mathcal{C}'$ is a refinement of $B_\mathcal{C}$. Similarly, $B_\mathcal{C}$ is a refinement of $B_\mathcal{C}'$. 
\end{proof}

\begin{definition}
Let $u \leq v$ be permutations in $S_n$,
 and let $\overline{T}(u):=\{t  \in T : u \lessdot ut \leq v\}$ 
and $\underline{T}(v):=\{t  \in T : v \gtrdot vt \geq u\}$
be the 
transpositions labeling the cover relations corresponding
to the atoms and coatoms in the interval. Define a labeled graph $G^{at}$ (resp. $G^{coat}$) on $[n]$ such that $G^{at}$ (resp. $G^{coat}$) has an edge
between $a$ and $b$ if and only if 
the transposition $(a b) \in \overline{T}(u)$ (resp. $(a b) \in \underline{T}(v)$). Let $B_{u,v}^{at}$
be the partition of $[n]$ whose
blocks  are the connected components of $G^{at}$.
Similarly, define partition $B_{u,v}^{coat}$ whose
blocks  are the connected components of $G^{coat}$.
\end{definition}

\begin{proposition} \label{atomPartition}
Let $[u,v] \subset S_n$.
The partitions $B_{u,v}^{at}$ and $B_{u,v}^{coat}$ 
are equal to $B_{u,v}$. Consequently, the labeled graphs $G^{\mathcal{C}}, G^{at}$ and $G^{coat}$ all have the same connected components.
\end{proposition}

\begin{proof}
The result follows from Corollary \ref{atomCoatomChains} and Corollary \ref{B}.
\end{proof}

We now prove Theorem \ref{th:dim}.
\begin{proof}[Proof of Theorem \ref{th:dim}]
We begin by showing that any point $(x_1,x_2,\ldots,x_n) \in \PPP_{u,v}$ satisfies the independent equations (\ref{affineSpaces}).
By Lemma \ref{wConst}, the linear functional $\omega=e_{B^j}$ is constant on $[u,v]$. Since $e_{B^j}(x_1,x_2,\ldots,x_n)=\sum_{i \in B^j} x_i$, \eqref{affineSpaces} holds.

Now suppose that there exists another affine space 
\begin{align} \label{auxEq}
\sum_{i=1}^n a_i x_i=c
\end{align}
 to which $\PPP_{u,v}$ belongs. By assumption, the linear functional $a=(a_1,\dots,a_n)$ is constant on $\PPP_{u,v}$, so by Lemma \ref{wConst}, 
\[
a=\sum_j c_j e_{B^j}
\]
for some coefficients $c_j$. Therefore equation (\ref{auxEq}) is a linear combination of equations (\ref{affineSpaces}).
\end{proof}

\begin{example}
Consider the intervals $[1234,1432]$ and $[1234,3412]$ in Figures \ref{interval1234to1432} and \ref{interval1234to3412}. We see that $B_{1234,1432}=|1|234|$ and $B_{1234,3412}=|1234|$, so that the dimensions are $2$ and $3$, respectively.

\centering
\begin{tikzpicture}[scale=.7] 
\node (u) at (0,-2) {$1234$};
  \node (a) at (-2,-0.1) {$1243$};
  \node (b) at (2,-0.1) {$1324$};
  \node (c) at (-2,2.1) {$1423$};
  \node (d) at (2,2.1) {$1342$};
  \node (v) at (0,4) {$1432$};
  \draw  (u) -- (a);
  \draw (u) -- (b); 
 \draw (b) -- (c);
 \draw (b) -- (d);
 \draw (a) -- (d);
  \draw (a) -- (c);
  \draw (b) -- (d);
  \draw (c) -- (v);
  \draw (d) -- (v);
  \node (x1) at (1.8,-1.2) {$(23)$};
  \node (x2) at (-1.8,-1.2) {$(34)$};
  \node (x3) at (-2.6,1) {$(23)$};
  \node (x4) at (2.6,1) {$(34)$};
  \node (x5) at (1.8,3.2) {$(23)$};
  \node (x6) at (-1.8,3.2) {$(34)$};
  \node (x7) at (0.75,2) {$(24)$};
  \node (x8) at (-0.75,2) {$(24)$};
\end{tikzpicture}
\captionof{figure}{}
\label{interval1234to1432}

\begin{center}
\begin{tikzpicture}[scale=.5] 
\node (u) at (0,-8) {\small{$1234$}};
\node (a) at (-4,-4) {\small{$1243$}};
\node (b) at (0,-4) {\small{$1324$}};
\node (c) at (4,-4) {\small{$2134$}};
\node (d) at (-6,0) {\small{$1423$}};
\node (e) at (-3,0) {\small{$1342$}};
\node (f) at (3,0) {\small{$3124$}};
\node (g) at (6,0) {\small{$2314$}};
\node (h) at (-6,4) {\small{$1432$}};
\node (i) at (-2,4) {\small{$2413$}};
\node (j) at (2,4) {\small{$3142$}};
\node (k) at (6,4) {\small{$3214$}};
\node (v) at (0,8) {\small{$3412$}};
\node (l) at (0,0) {\small{$2143$}};
\draw  (u) -- (a);
\draw  (l) -- (j);
\draw  (l) -- (i);
\draw  (l) -- (c);
\draw  (l) -- (a);
\node (x1) at (-3,-6) {\tiny{$(34)$}};
\draw  (u) -- (b);
\node (x2) at (-0.6,-6) {\tiny{$(23)$}};
\draw  (u) -- (c);
\node (x3) at (1.3,-6) {\tiny{$(12)$}};
\draw  (a) -- (d);
\node (x4) at (-5.5,-2) {\tiny{$(23)$}};
\draw  (a) -- (e);
\draw  (b) -- (d);
\draw  (b) -- (e);
\draw  (b) -- (f);
\draw  (b) -- (g);
\draw  (c) -- (f);
\draw  (c) -- (g);
\draw  (d) -- (h);
\draw  (d) -- (i);
\node (x5) at (-3.8,3) {\tiny{$(13)$}};
\draw  (e) -- (h);
\draw  (e) -- (j);
\draw  (f) -- (j);
\draw  (f) -- (k);
\draw  (g) -- (i);
\draw  (g) -- (k);
\draw  (h) -- (v);
\node (x6) at (-3.8,6) {\tiny{$(13)$}};
\draw  (i) -- (v);
\node (x7) at (-1.7,6) {\tiny{$(14)$}};
\draw  (j) -- (v);
\node (x8) at (1.6,6) {\tiny{$(23)$}};
\draw  (k) -- (v);
\node (x8) at (3.6,6) {\tiny{$(24)$}};
\end{tikzpicture}
\captionof{figure}{}
\label{interval1234to3412}
\end{center}

\end{example}

We now turn to the question of when the Richardson variety 
$\mathcal{R}_{u,v}$ is a toric variety.  Our proof uses Proposition \ref{prop:toric},
which will be proved later, using properties of the moment map.

\begin{proposition}\label{RichardsonIsToric}
The Richardson variety $\mathcal{R}_{u,v}$ in $\Fl_n$ is a toric variety if and only if 
the number of blocks $\# B_{u,v}$ of the partition $B_{u,v}$ satisfies 
$\# B_{u,v} = n-\ell(v)+\ell(u)$.
Equivalently, $\mathcal{R}_{u,v}$ is a toric variety if and only if the labeled graph $G^{\mathcal{C}}$ is a forest (with no multiple edges).
\end{proposition}

\begin{proof}
By Proposition \ref{prop:toric}, 
$\mathcal{R}_{u,v}$ is a toric variety if and only if 
$\dim \mathsf{Q}_{u,v} = \ell(v)-\ell(u)$.
The first statement of the proposition now follows from 
Theorem \ref{th:dim}.

We will prove the second statement from the first.  
Note that $\mathcal{C}$ is a chain with $\ell(v)-\ell(u)$ edges.  
Let us consider the process of building the graph $G$
by adding one edge at a time while reading the edge-labels of $\mathcal{C}$, say from top to bottom.
We start out with a totally disconnected graph on the vertices $[n]$.  Adding a new edge will either 
preserve the number of connected components of the graph, or will decrease it by $1$.  
In order to arrive at a partition $B_{u,v}$ with $n-(\ell(v)-\ell(u))$ parts, we must 
decrease the number of connected components of the graph with every new edge added.  
But this will happen if and only if the graph $G$ we construct is a forest
(with no multiple edges).
\end{proof}

Given a labeled graph $G$, we will say that a cycle $(v_0,v_1,\ldots,v_k)$ with $v_k=v_0$ is \emph{increasing} if $v_0<v_1<\ldots<v_{k-1}$. We shall call a labeled graph with no increasing cycles an \emph{increasing-cycle-free labeled graph}.

\begin{lemma} \label{triangleFree}
The labeled graphs $G^{at}$ and $G^{coat}$ are increasing-cycle-free. In particular, they are simple and triangle-free.
\end{lemma}

\begin{proof}
From the definition, it is clear that the graphs are simple. Assume by contradiction that $C=(v_0,v_1,\ldots,v_k)$ is an increasing cycle in $G^{at}$. By properties of Bruhat order on the symmetric group, the existence of an edge $\{a,b\}$ with $a<b$ implies that $u(a)<u(b)$ and for any $a < c < b$, $u(c) \not \in [u(a),u(b)]$. Looking at edges $\{v_{i} v_{i+1}\}, i=0,1,\ldots,k-2$, of cycle $C$, we see that
\[
u(v_0)<u(v_1)<\ldots<u(v_{k-1}).
\]
However, edge $\{v_0,v_{k-1}\}$ implies that $u(v_i) \not \in [u(v_0),u(v_{k-1})]$ for any $1 \leq i \leq k-2$, which is a contradiction.
The proof for $G^{coat}$ is analogous.
\end{proof}

Following Bj\"{o}rner and Brenti \cite{BB}, we call the face poset of a $k$-gon a $k$-crown. Any length $3$ interval in a Coxeter group is a $k$-crown \cite[Corollary 2.7.8]{BB}. It is also known that in $S_n$, the values of $k$ can only be $2,3$ or $4$.

\begin{remark}
Using Proposition \ref{atomPartition} and Lemma \ref{triangleFree}, it is easy to show that any $k$-crown must have $k \leq 4$. Indeed, the graph $G^{\mathcal{C}}$ has $3$ edges, and therefore at least $n-3$ connected components. By Proposition \ref{atomPartition}, the graph $G^{at}$ has the same connected components as $G^{\mathcal{C}}$ and $k$ edges. By Lemma \ref{triangleFree} it is simple and triangle-free. Consequently, if $k>4$ then $G^{at}$ must have at most $n-4$ components. 
\end{remark}

\begin{lemma} \label{4crown}
Let $[u,v]$ be a $4$-crown and let $\mathcal{C}:u=x_{(0)} \lessdot x_{(1)} \lessdot x_{(2)} \lessdot   x_{(3)}=v$ be any maximal chain. The graph $G^{\mathcal{C}}$ is a forest. In particular, if we set 
$t_i:=x_{(i)}^{-1} x_{(i+1)}$ for
$0 \leq i \leq 2$, then $t_0 \neq t_2$ since there are no multiple edges.
\end{lemma}

\begin{proof}
The graph $G^{at}$ has $4$ edges. By the discussion above, the smallest cycle $G^{at}$ can have is of length $4$. Therefore $G^{at}$ has at most $n-3$ connected components. 

Assume by contradiction that $G^{\mathcal{C}}$ is not a forest. Then the graph $G^{\mathcal{C}}$, which has $3$ edges, has at least $n-2$ connected components. But the number of connected components of $G^{\mathcal{C}}$ and $G^{at}$ must be equal, so we obtain a contradiction.
\end{proof}

\begin{corollary}
A Richardson variety $\mathcal{R}_{u,v}$ in $\Fl_n$ with $\ell(v)-\ell(u)=3$ is a toric variety if and only if $[u,v]$ is a $3$-crown or a $4$-crown. 
\end{corollary}

\begin{proof}
The interval $[u,v]$ is a $k$-crown for $k=2,3$ or $4$. If $k=2$ or $3$, then, by Lemma \ref{triangleFree} the graph $G^{at}$ must have $n-k$ connected components. Consequently, $k$ cannot equal to $2$. For $k=3$, we see that the Richardson variety is toric. For $k=4$, the result follows from Lemma \ref{4crown}.
\end{proof}

\subsection{Faces of Bruhat interval polytopes}

Using the results of prior sections, we will give a combinatorial criterion 
for when one Bruhat interval polytope is a face of another
(see Theorem \ref{faceCrit}).  First we need a few lemmas.

Let $\overline{T}(x,X):=\{t  \in T : \exists z \overset{t}{\gtrdot} x, z \in X \}$  and $\underline{T}(x,X):=\{t  \in T : \exists z \overset{t}{\lessdot} x, z \in X \}$
be the transpositions labeling cover relations of an element $x$ in a set $X$. In the following we use the convention that $i<j$ in $(i,j)$.

\begin{lemma}\label{diamondineq}
Let $\omega: \R^n \rightarrow \R$ be a linear functional satisfying
\[
\omega(b) = \omega(a) \geq \omega(d)
\]
for $a,b,c,d \in S_n$ the elements of a Bruhat interval of length $2$:

\begin{center}
\begin{tikzpicture}[scale=.5] 
\node (u) at (0,-3) {{$a$}};
  \node (a) at (-3,-0.1) {{$c$}};
  \node (b) at (3,-0.1) {{$b$}};
  \node (v) at (0,3) {{$d$}};
  \draw  (u) -- (a);
  \draw (u) -- (b); 
  \draw (a) -- (v);
  \draw (b) -- (v);
  
\end{tikzpicture}
    \end{center}
Then either $\omega(c) \leq \omega(a)$ or $\omega(c) \leq \omega(d)$.    
\end{lemma}

\begin{proof}
By Theorem \ref{GeneralizedLiftingProperty}, there exists $(i,j)$ such that either
\[
a \overset{(i,j)}{\lessdot} b \text{ and } c \overset{(i,j)}{\lessdot} d
\]
or
\[
a \overset{(i,j)}{\lessdot} c \text{ and } b \overset{(i,j)}{\lessdot} d.
\]
In the former case, $\omega(a) = \omega(b) \iff \omega(c) = \omega(d)$. In the latter case, $\omega(a) \geq \omega(c) \iff \omega(b) \geq \omega(d)$.
\end{proof}

\begin{lemma} \label{coversuv}
Let $x,y,u,v \in S_n$ such that $u \leq x \leq y \leq v$. Assume that $\omega: \R^n \rightarrow \R$ is a linear functional satisfying 
\begin{align}
w:=\omega(z) \text{ for all } z \in [x,y],
\end{align}
\begin{align}\label{eq:1}
\omega(y) \geq \omega(b) \text{ for all } b \gtrdot y, b \in [u,v],
\end{align}
and
\begin{align}\label{eq:2}
\omega(x)\geq \omega(a) \text{ for all } a \lessdot x, a \in [u,v].
\end{align}
Then for any $z \in [x,v] \cup [u,y]$, $\omega(z) \leq w$.
\end{lemma}

\begin{proof}
By Corollary \ref{atomCoatomChains}, the result holds for any $z \in [y,v] \cup [u,x]$. Indeed, given a $z \in [y,v]$, we can construct a chain from $y$ to $z$ whose transpositions are in $\overline{T}(y,[y,z]) \subset \overline{T}(y,[y,v])$. Analogously, for $z \in [u,x]$, we can construct a chain from $z$ to $x$ with transpositions in  $\underline{T}(x,[z,x]) \subset \underline{T}(x,[u,x])$.
Applying \eqref{eq:1} and \eqref{eq:2}, respectively, yields the result.

Now let $q \in [x,y]$. We show that
\[
\omega(q) \geq \omega(z) \, \forall z \gtrdot q, z \in [u,v].
\]
Proceed by induction on $m:=\ell(y)-\ell(q) \geq 0$. The base case holds by assumption. Consider now such a $q$ with $m \geq 1$. Suppose that $z\in [u,v]$,
where $z \gtrdot q$, and take $q' \gtrdot q$ with $q' \in [x,y]$. We have the following diagram
\begin{center}
\begin{tikzpicture}[scale=.5] 
\node (u) at (0,-3) {{$q$}};
  \node (a) at (-3,-0.1) {{$z$}};
  \node (b) at (3,-0.1) {{$q'$}};
  \node (v) at (0,3) {{$z'$}};
  \draw  (u) -- (a);
  \draw (u) -- (b); 
  \draw (a) -- (v);
  \draw (b) -- (v);
  
\end{tikzpicture}
    \end{center}
for some $z' \gtrdot z,q'$, with $z' \in [u,v]$. The existence of such a $z'$ comes from the fact that Bruhat order is a directed poset \cite[Proposition 2.2.9]{BB} and from the structure of its length $2$ intervals \cite[Lemma 2.7.3]{BB}.
 By induction, $\omega(q') \geq \omega(z')$. We also know that $\omega(q)=\omega(q')$. Applying Lemma \ref{diamondineq} completes the induction. 

We have shown in particular that $\omega(x) \geq \omega(z)$ for all $x \lessdot z \in [u,v]$. Applying Corollary \ref{atomCoatomChains} shows that $\omega(x) \geq \omega(z) \, \forall z \in [x,v]$.
By symmetry, $\omega(y) \geq \omega(z) \, \forall z \in [u,y]$.
\end{proof}

\begin{theorem}\label{faceCrit}
Let $[x,y] \subset [u,v]$.   We define the graph 
$G_{x,y}^{u,v}$ as follows:
\begin{enumerate}
\item The nodes of $G_{x,y}^{u,v}$ are $\{1,2,\ldots,n\}$,
 with nodes $i$ and $j$ identified if they are in the same part of $B_{x,y}$.
\item There is a directed edge from $i$ to $j$ for every $(i,j) \in \overline{T}(y,[u,v])$.
\item There is a directed edge from $j$ to $i$ for every $(i,j) \in \underline{T}(x,[u,v])$.
\end{enumerate}
Then the Bruhat interval polytope $\PPP_{x,y}$ is a face of the Bruhat interval polytope $\PPP_{u,v}$ if and only if the graph 
$G_{x,y}^{u,v}$ 
is a directed acyclic graph.
\end{theorem}

\begin{proof}
Assume first that $\omega: \R^n \rightarrow \R$ is a linear functional with $\omega |_{\PPP_{u,v}}$ maximized exactly on $\PPP_{x,y}$. Since 
$\omega$ is constant on $[x,y]$, $\omega$ is compatible with the partition $B_{x,y}$, i.e. $\omega_i = \omega_j$ whenever $i$ and $j$ are in the same 
part of $B_{x,y}$. From the definition of $\omega$,
\begin{align*}
&\omega(y)>\omega(b) \text{ for all } b \gtrdot y, b \in [u,v], \text{ and }\\
&\omega(x)>\omega(a) \text{ for all } a \lessdot x, a \in [u,v].
\end{align*}
Equivalently,
\begin{align*}
&\omega_i <\omega_j \text{ for all } (i,j) \in \overline{T}(y,[u,v]), \text{ and }\\
&\omega_i >\omega_j \text{ for all } (i,j) \in \underline{T}(x,[u,v]).
\end{align*}
Label each vertex $k \in G_{x,y}^{u,v}$ with the number $\omega_k$. If $G_{x,y}^{u,v}$ has  a directed edge from $i$ to $j$, then $\omega_i < \omega_j$. It follows that $G_{x,y}^{u,v}$ is acyclic.

Conversely, we assume that $G_{x,y}^{u,v}$ is acyclic. Consequently, there exists a linear ordering $L$ of the vertices of $G_{x,y}^{u,v}$ such that for every directed edge $i\rightarrow j$ from vertex $i$ to vertex $j$, $i$ comes before $j$ in the ordering (i.e., $i \rightarrow j \implies L(i)<L(j)$). Define $\omega:\R^n \rightarrow \R$ via
\[
\omega_i:=L(i).
\]

Since $\omega$ is constant on each block of $B_{x,y}$, 
$\omega$ is constant on $[x,y]$.  Also, 
\begin{align}\label{ineqsOmega}
&\omega_i <\omega_j \text{ for all } (i,j) \in \overline{T}(y,[u,v]), \text{ and } \\
&\omega_i >\omega_j \text{ for all } (i,j) \in \underline{T}(x,[u,v]),\label{ineqsOmega2}
\end{align}
so that
\begin{align*}
&\omega(y)>\omega(b) \text{ for all } b \gtrdot y, b \in [u,v], \text{ and }\\
&\omega(x)>\omega(a) \text{ for all } a \lessdot x, a \in [u,v].
\end{align*}

We show now that these conditions imply that $\omega$ defines $\PPP_{x,y}$ as a face of $\PPP_{u,v}$. Indeed, note that $y$ is a vertex of $\PPP_{u,v}$, and by Theorem \ref{bipFaceisBIP}, any edge of $\PPP_{u,v}$ emanating from $y$ corresponds to a cover relation of $y$. Thus, if $f$ is an edge vector emanating from $y$, then $y+f=z$ for some $z \gtrdot y$ or $z \lessdot y$ in $[u,v]$. Consequently, by Lemma \ref{coversuv},
\[
\omega(z) \leq \omega(y) \implies \omega(f) \leq 0.
\]
This argument shows that $\omega(f) \leq 0$ for any edge vector. By convexity, $\PPP_{u,v}$ is contained in the polyhedral cone spanned by the edges emanating from $y$. Therefore $\omega(y) \geq \omega(z) \, \forall z \in [u,v]$. Similarly, $\omega(x) \geq \omega(z) \, \forall z \in [u,v]$. It follows that $[x,y]$ is a subset of the face defined by $\omega$. By Theorem \ref{bipFaceisBIP}, this face corresponds to an interval, which we showed contains $[x,y]$. Inequalities (\ref{ineqsOmega}), (\ref{ineqsOmega2}) imply that this interval is no larger than $[x,y]$.
\end{proof} 

From the proof we obtain

\begin{corollary}
The normal cone of $\PPP_{x,y}$ in $\PPP_{u,v}$ is the set of linear functionals $\omega=(\omega_i)$ compatible with $G_{x,y}^{u,v}$:
\begin{enumerate}
\item $\omega_i=\omega_j$ if $i,j$ are identified nodes of $G_{x,y}^{u,v}$,
\item $\omega_i < \omega_j$ if there is a directed edge from $i$ to $j$ in $G_{x,y}^{u,v}$.
\end{enumerate}
\end{corollary}

\begin{example}

Set $[u,v]=[1243,4132]$.
Let us verify using Theorem \ref{faceCrit} that the BIP $\PPP_{2143,4132}$ corresponding to $[x,y]=[2143,4132]$ is a face of $\PPP_{1243,4132}$.
\begin{figure}[H]
\centering
\includegraphics[width=0.4\linewidth]{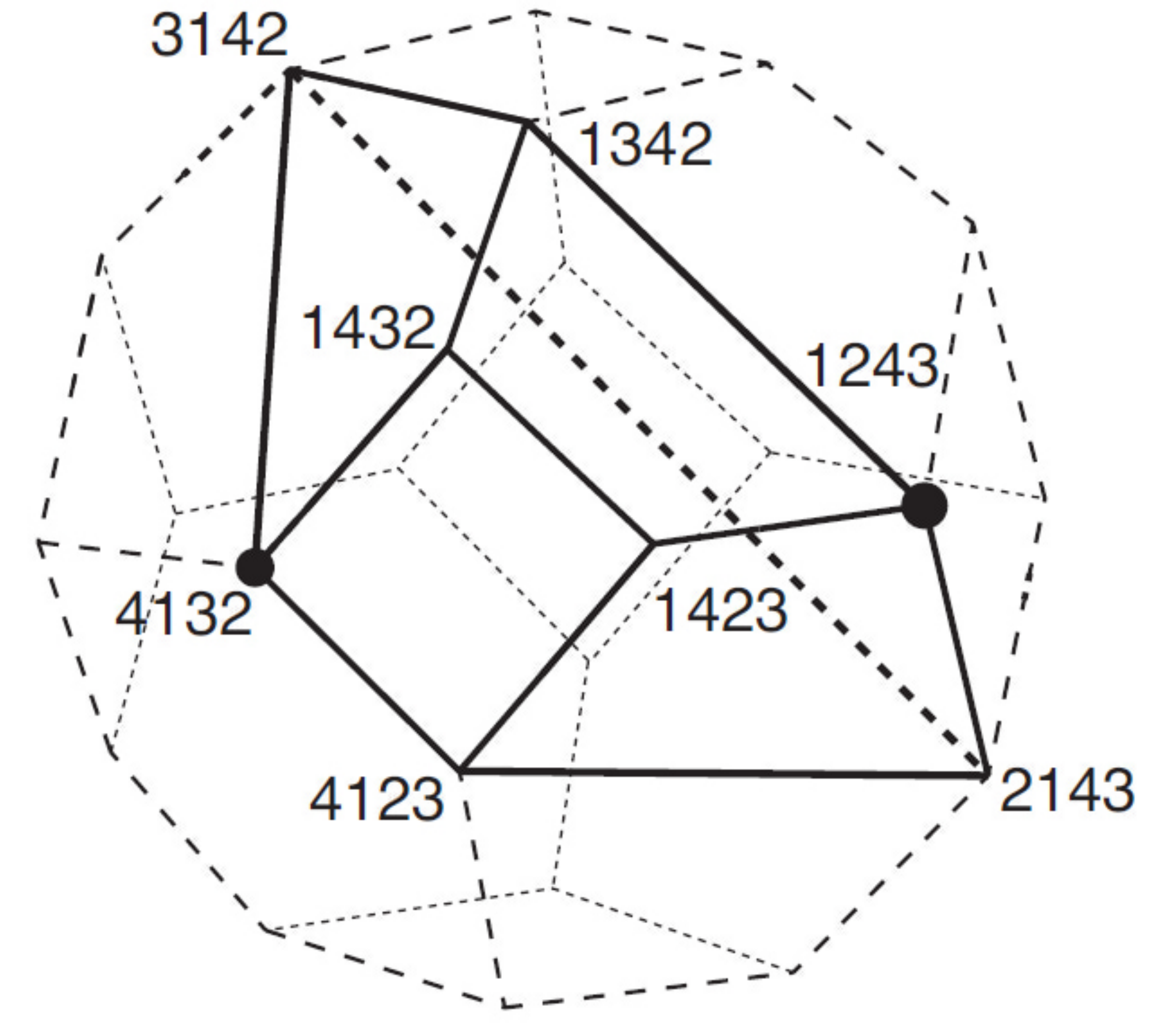}
\includegraphics[width=0.4\linewidth]{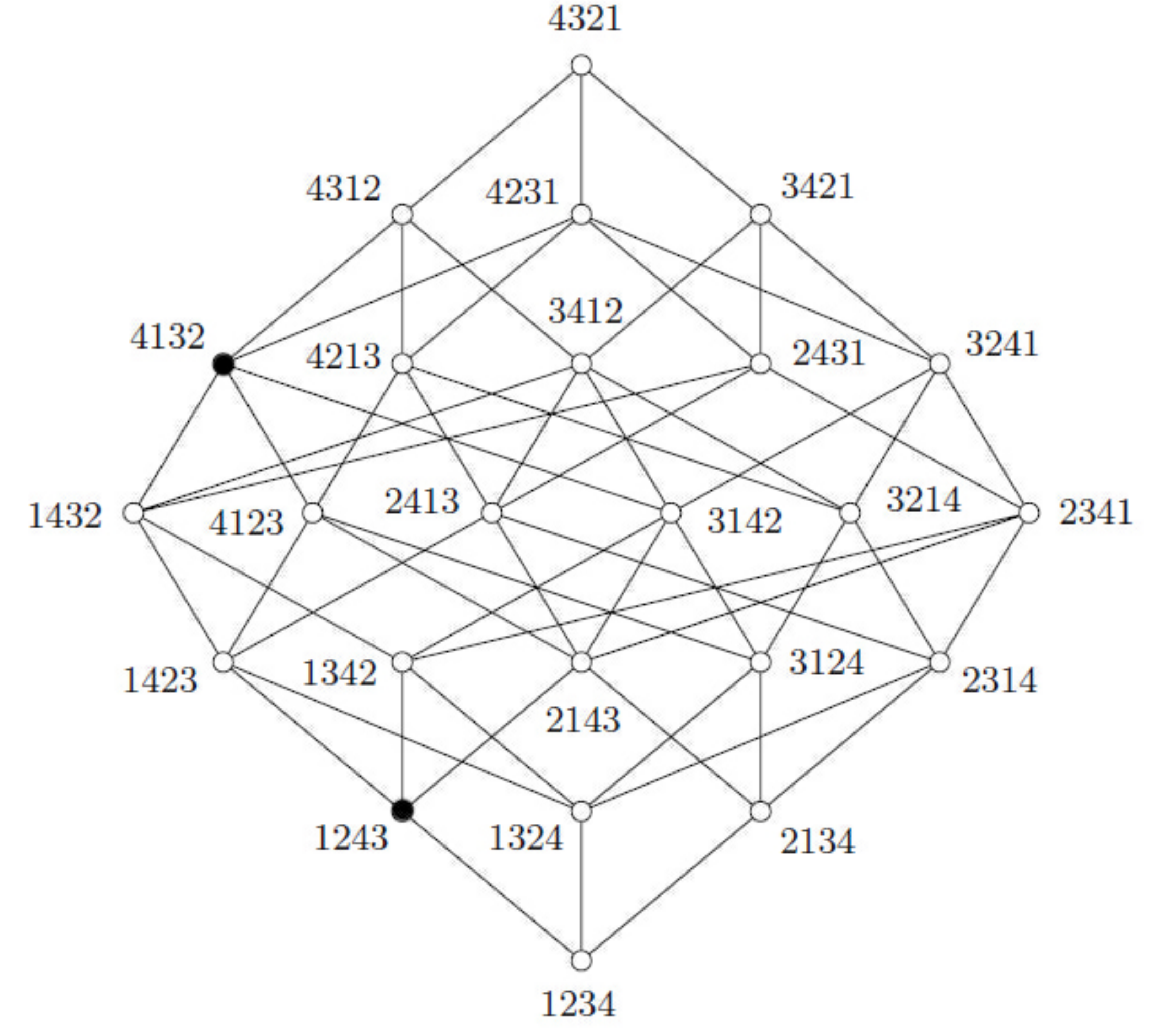}
\caption*{\label{bruh} $\PPP_{1243,4132}$}
\end{figure}

The interval $[x,y]$ along with its neighbors in $[u,v]$ are

\begin{center}
\begin{tikzpicture}[scale=.5] 
\node (u) at (0,-3) {{$2143$}};
  \node (a) at (-3,-0.1) {{$4123$}};
  \node (c) at (0,-6) {{$1243$}};
  \node (b) at (3,-0.1) {{$3142$}};
  \node (v) at (0,3) {{$4132$}};
  \draw  (u) -- (a);
   \draw  (u) -- (c);
  \draw (u) -- (b); 
  \draw (a) -- (v);
  \draw (b) -- (v);
  \node (x1) at (2.5,-1.8) {{$(14)$}};
  \node (x4) at (2.5,1.8) {{$(13)$}};
  \node (x2) at (0.8,-4.5) {{$(12)$}};
\end{tikzpicture}
    \end{center}
    
Therefore the graph $G_{x,y}^{u,v}$ is
    
    \begin{center}
   \begin{tikzpicture}[scale=.7,main node/.style={circle,fill=blue!20,draw,font=\sffamily\Large\bfseries}] 
\node[main node] (u) at (0,0) {{$1,3,4$}};
  \node[main node] (a) at (3,0) {{$2$}};
   \path (a) edge [->] (u);
\end{tikzpicture}
 \end{center}
 which is clearly acyclic.

\end{example}

\subsection{Diameter of Bruhat interval polytopes}
In this section, we show that the diameter of $\PPP_{u,v}$ is equal to $\ell(v)-\ell(u)$ (see Theorem \ref{diameter}). Let 
\[
\overline{E}(x,[u,v]):=\{z-x \in \R^n : z \gtrdot x, z \in [u,v]\}, \quad \underline{E}(x,[u,v]):=\{z-x \in \R^n : z \lessdot x, z \in [u,v]\}.
\]
We note that $\overline{E}(x,[u,v]),\underline{E}(x,[u,v]) \subset [u,v]-x$.

For a set $\{v_1,\ldots,v_m\} \subset \R^n$, define
\[
\cone(\{v_1,\ldots,v_m\}):=\R_{\geq 0} v_1 +\ldots+\R_{\geq 0} v_m.
\]

\begin{lemma} \label{cones}
If $e \in \underline{E}(x,[u,v])$, then $e \not \in \cone(\overline{E}(x,[u,v]))$. Similarly, if $e \in \overline{E}(x,[u,v])$, then $e \not \in \cone(\underline{E}(x,[u,v]))$.
\end{lemma}

\begin{proof}
Recall that a set of simple roots for type $A$ is given by
\[
e_i-e_{i+1}, \quad i=1,2,\ldots,n-1,
\]
and that any other root can be expressed uniquely as a linear combination of simple roots with integral coefficients of the same sign.

Let $e \in \underline{E}(x,[u,v])$. Then $e$ is of the form
\[
e=c(e_j-e_i), \quad i<j, c>0, 
\]
which is in the cone of negative roots. On the other hand, $\cone(\overline{E}(x,[u,v]))$ is a subset of the cone of positive roots. It follows that $e \not \in \cone(\overline{E}(x,[u,v]))$. The argument for $e \in \overline{E}(x,[u,v]) \implies e \not \in \cone(\underline{E}(x,[u,v]))$ is analogous. 
\end{proof}

\begin{lemma}\label{edgeLemma}
Let $x \in (u,v) \subset S_n$. The sets $\overline{E}(x,[u,v])$ and $\underline{E}(x,[u,v])$ each contain an edge of $\PPP_{u,v}$ incident to $x$.
\end{lemma}

\begin{proof}
As a consequence of Theorem \ref{bipFaceisBIP}, the edges of $\PPP_{u,v}$ incident to $x$ are a subset of $\overline{E}(x,[u,v]) \cup \underline{E}(x,[u,v])$. Assume by contradiction that all edges of $\PPP_{u,v}$ incident to $x$ are in $\overline{E}(x,[u,v])$. Then by convexity of $\PPP_{u,v}$, 
\[
\PPP_{u,v} \subset x+\cone(\overline{E}(x,[u,v])) \implies \cone(\PPP_{u,v}-x) \subset \cone(\overline{E}(x,[u,v])).
\]
Since $\cone(\underline{E}(x,[u,v])) \subset \cone(\PPP_{u,v}-x)$, we have $\cone(\underline{E}(x,[u,v])) \subset \cone(\overline{E}(x,[u,v]))$. By assumption, $x \not \in \{u,v\}$, so that $\underline{E}(x,[u,v]) \neq \emptyset$,  contradicting Lemma \ref{cones}. The argument for $\underline{E}(x,[u,v])$ is analogous.
\end{proof}

\begin{theorem}\label{diameter}
The diameter of $\PPP_{u,v}$ is equal to $\ell(v)-\ell(u)$.
\end{theorem}

\begin{proof}
Since any edge of a Bruhat interval polytope corresponds to a cover relation in Bruhat order, the distance from $u$ to $v$ is at least $\ell(v)-\ell(u)$. By Lemma \ref{edgeLemma}, there is a path from $u$ to $v$ which takes steps up in the Hasse diagram. Therefore the distance from $u$ to $v$ is $\ell(v)-\ell(u)$.

Next, take $x \neq y$ in $[u,v]$. By Lemma \ref{edgeLemma}, a path from $x$ to $y$ can be formed by either taking a path of length $\ell(v)-\ell(x)$ from $x$ to $v$ and then of length $\ell(v)-\ell(y)$ from $v$ to $y$, or of length $\ell(x)-\ell(u)$ from $x$ to $u$ and then of length $\ell(y)-\ell(u)$ from $u$ to $y$. One of these paths must be of length less than or equal to $\ell(v)-\ell(u)$, since the sum of the lengths is $2(\ell(v)-\ell(u))$.
\end{proof}

\subsection{An inequality description of Bruhat interval polytopes}

Using Proposition \ref{prop:Mink-sum2}, 
which says that Bruhat interval polytopes 
are Minkowski 
sums of matroid polytopes, we will provide an inequality description
of Bruhat interval polytopes.

We first need to recall the notion of the rank function $r_{\M}$ of a matroid $\M$.
Suppose that $\M$ is a matroid of rank $k$ on the ground set $[n]$.  Then 
the \emph{rank function} $r_{\M}: 2^{[n]} \to \Z_{\geq 0}$ is the function defined by 
\begin{equation*}
r_{\M}(A) = \max_{I \in \M} |A \cap I| \text{ for all }A \in 2^{[n]}.
\end{equation*}

There is an inequality description of matroid polytopes, using the rank function.
\begin{proposition}[\cite{welsh}]\label{r:inequalitiesmatroids}
Let $\M$ be any matroid of rank $k$ on the ground set $[n]$, and 
let $r_{\M}:2^{[n]} \to \Z_{\geq 0}$ be its rank function. Then the matroid polytope $\Gamma_{\M}$ can be described as
\[
\Gamma_{\M} = \left\{ {\bf x} \in \R^n \ \vert \ \sum_{i \in [n]} x_i = k, \, \sum_{i \in A} x_i \leq r_{\M}(A) \, \text{ for all $A \subset [n]$} \right\}.
\]
\end{proposition}

Using Proposition \ref{r:inequalitiesmatroids} we obtain the following result.

\begin{proposition}\label{prop:inequality}
Choose $u \leq v \in S_n$, and for each $1 \leq k \leq n-1$,
define the matroid 
\begin{equation*}
\M_k = \{I \in {[n] \choose k} \ \vert \ \text{ there exists }
z \in [u,v] \text{ such that }I=\{z(1),\dots, z(k)\}\}.
\end{equation*}
Then 
\begin{equation*}
\mathsf{Q}_{u,v} =  \left\{ {\bf x} \in \R^n \ \vert \ \sum_{i \in [n]} x_i = {n+1 \choose 2}, \, 
\sum_{i \in A} x_i \leq \sum_{j=1}^{n-1} r_{\M_j}(A) \, \text{ for all $A \subset [n]$} \right\}.
\end{equation*}
\end{proposition}

\begin{proof}
We know from Proposition \ref{prop:Mink-sum2} that 
$\mathsf{Q}_{u,v}$ is the Minkowski sum 
$$\mathsf{Q}_{u,v} = \sum_{k=1}^{n-1} \Gamma_{\M_k},$$ where
$\M_k$ is defined as above.
But now Proposition \ref{prop:inequality} 
follows from 
Proposition \ref{r:inequalitiesmatroids}
and the observation (made in the proof of \cite[Lemma 2.1]{ABD}) that, if a linear functional $\omega$ takes maximum values
$a$ and $b$ on (faces $A$ and $B$ of) polytopes $P$ and $Q$, respectively, then
it takes maximum value $a+b$ on (the face $A+B$ of) their Minkowski sum.
\end{proof}

\begin{example}
Consider $u=1324$ and $v=2431$ in $S_4$.  
We will compute the inequality description of $\mathsf{Q}_{u,v}$.
First note that $[u,v] = \{1324, 1342, 1423, 1432, 2314, 2341, 2413, 2431\}$.
We then compute:
\begin{itemize} 
\item $\M_1 = \{\{1\}, \{2\}\}$,  a matroid of rank $1$ on $[4]$.
\item $\M_2 = \{\{1,3\}, \{1,4\}, \{2,3\}, \{2,4\}\}$, a matroid of rank $2$ on $[4]$.
\item $\M_3 = \{\{1,2,3\}, \{1,2,4\}, \{1,3,4\}, \{2,3,4\}\}$, a matroid of rank $3$ on $[4]$.
\end{itemize}
Now using Proposition \ref{prop:inequality}, we get 
\begin{align*}
\mathsf{Q}_{u,v}  = 
\{ {\bf x} \in \R^4 \ \vert \  \sum_{i \in [4]} x_i = 10, \, 
& x_1+x_2+x_3 \leq 6, x_1 + x_2 + x_4 \leq 6,
 x_1+x_3 + x_4 \leq 6, x_2 + x_3 + x_4 \leq 6,\\
& x_1+x_2 \leq 4, x_1+x_3 \leq 5, x_1 + x_4 \leq 5,
 x_2+x_3 \leq 5, x_2+x_4 \leq 5,
 x_3+x_4 \leq 3, \\
& x_1 \leq 3, x_2\leq 3, x_3 \leq 2, x_4 \leq 2.\}
\end{align*}

\end{example}

\section{A generalization of the recurrence for $R$-polynomials}\label{sec:R}

The well-known $R$-polynomials were introduced by Kazhdan and Lusztig  as a useful tool for computing
Kazhdan-Lusztig polynomials \cite{KL}.  $R$-polynomials also have a
geometric interpretation in terms of Richardson varieties.  More specifically,
the Richardson variety 
$\mathcal{R}_{u,v}$  (see Section \ref{sec:preliminaries} for the definition)
may be defined over a finite field $\F_q$, and the number of points
it contains is given by the $R$-polynomial $R_{u,v}(q) = \#\mathcal{R}_{u,v}(\F_q)$.

The $R$-polynomials may be defined by the following recurrence.
\begin{theorem} \cite[Theorem 5.1.1]{BB}\label{prop:R}
There exists a unique family of polynomials $\{R_{u,v}(q)\}_{u,v \in W} \subset \Z[q]$ satisfying the following conditions:
\begin{enumerate}
\item $R_{u,v}(q)=0 ,\ \text{ if }u \not \leq v$.
\item $R_{u,v}(q)=1 ,\ \text{ if }u=v$.
\item If $s \in D_R(v)$, then 

\[ R_{u,v}(q)=\begin{cases} {R}_{us,vs}(q) & \text{ if }s \in D_R(u),\\
                                            
        q {R}_{us,vs}(q)+(q-1){R}_{u,vs}(q) & \text{ if }s \not \in D_R(u)

        .\end{cases}
\]
\end{enumerate}
\end{theorem}

It is natural to wonder whether one can replace $s$ with a transposition $t$ whenever the Generalized lifting property holds. More precisely, suppose that $t$ is a transposition such that
\begin{align} \label{GeneralLiftingShort}
vt \lessdot v \quad
u \lessdot ut \quad
u \leq vt \quad
ut \leq v.
\end{align}
Is it true that 
\begin{equation}\label{eq:Rpoly}
{R}_{u,v}(q)=q {R}_{ut,vt}(q)+(q-1){R}_{u,vt}(q)?
\end{equation}
In general, the answer is no. For example, one can check that $u=1324$, $v=4231$ and $t=(24)$ give a counterexample. 
However, 
when $t$ is an inversion-minimal transposition on $(u,v)$, 
\eqref{eq:Rpoly} does hold.
We'll use the next lemma to prove this.

\begin{lemma} \label{minimalIntermediate}
Let $u,v \in S_n$ and suppose that $(i k)$ is inversion-minimal on $(u,v)$. Assume further that  $v_j>v_{j+1}$ and $u_j>u_{j+1}$ for some $j$ such 
that $i<j<k-1$. Then $(i k)$ is inversion-minimal on $(vs_j,us_j)$.
\end{lemma}

\begin{proof}
The result follows directly from the definition.
\end{proof}

\begin{proposition} \label{Rpolyt}
Let $u,v \in S_n$ with $v \geq u$. Let $t=(i j)$ be inversion-minimal on $(u,v)$. Then
\[
 {R}_{u,v}(q)=q{R}_{ut,vt}(q)+(q-1){R}_{u,vt}(q).
\]
\end{proposition}

\begin{proof}
Proceed by induction on $\ell=j-i$. The base case $\ell=1$ follows from Theorem \ref{prop:R}. Assume the inductive hypothesis and consider $\ell>1$. Since $(i j)$ is inversion-minimal on $(u,v)$, we have $v_i>v_j$ and $u_i<u_j$.

\textbf{Case 1}: Suppose that $v_i>v_{i+1}$ and $u_i>u_{i+1}$ or $v_i<v_{i+1}$ and $u_i<u_{i+1}$.

We have ${R}_{u,v}(q)={R}_{us_i,vs_i}(q)$. Let $t'$ be the 
transposition $((i+1)j)$.  Clearly $t'$ 
is inversion-minimal on $(vs_i,us_i)$. By induction,
\[
{R}_{us_i,vs_i}(q)=q{R}_{us_i t',vs_i t'}(q)+(q-1) {R}_{us_i,vs_i t'}(q).
\]
By Lemma \ref{lem:dot-bijection}, we have $v_{i+1}>v_j \iff u_{j+1}>u_j$.
Using this and the fact that $s_i t' s_i=(ij)=t$, we see that
\[
{R}_{us_i t',vs_i t'}(q)=R_{ut s_i, v t s_i}(q) = {R}_{ut,vt}(q).
\]
Similarly, by Lemma \ref{lem:dot-bijection}, we have
$u_{i+1}>u_i \iff v_{i+1}>v_i \iff v_{i+1}>v_j$.  This implies that 
\[
{R}_{us_i,vs_i t'}(q)={R}_{u,vt}(q).
\]
Putting everything together, we have the desired equality
\[
R_{u,v}(q) = R_{us_i, vs_i}(q) = R_{ut, vt}(q)+R_{u,vt}(q).
\]

\textbf{Case 2}: Suppose that $v_{j-1}>v_{j}$ and $u_{j-1}>u_j$ or $v_{j-1}<v_j$ and $u_{j-1}<u_j$.

This case is analogous to Case 1.

\textbf{Case 3}: Suppose that neither of the above two cases holds.

Since $(i j)$ is inversion-minimal on $(u,v)$, we must have $v_i<v_{i+1}$ and $v_{j-1}<v_j$. Since $v_i>v_j$, there exists some $m_1 \in (i,j-1)$ such that $v_{m_1}>v_{m_1+1}$. Using the fact that $(i j)$ is inversion-minimal on 
$(u,v)$, we must have  $u_{m_1}>u_{m_1+1}$. 
By Lemma \ref{minimalIntermediate}, $(i j)$ is inversion-minimal on $(vs_{m_1},us_{m_1})$. If $us_{m_1}$ and $vs_{m_1}$ do not satisfy the conditions of Cases 1 or 2, then we may find $m_2 \in (i,j-1)$ and then $(i j)$ is inversion-minimal on $(vs_{m_1} s_{m_2},us_{m_1} s_{m_2})$. Such a sequence $m_1,m_2,\ldots$ clearly terminates. Assume that it terminates at $k$, so that $(i j)$ is inversion-minimal on $(vs_{m_1} s_{m_2} \cdots s_{m_k},u s_{m_1}s_{m_2} \cdots s_{m_k})$ and the hypotheses of Case 1 or 2 are satisfied for $vs_{m_1} s_{m_2} \cdots s_{m_k}$ and $u s_{m_1}s_{m_2} \cdots s_{m_k}$. Set $\Pi_k:=s_{m_1} s_{m_2} \cdots s_{m_k}$. We then have

\[
{R}_{u\Pi_k,v\Pi_k}(q)={R}_{u \Pi_k t,v \Pi_k t}(q)+q{R}_{u\Pi_k,v\Pi_kt}(q)
\]

To prove Proposition \ref{Rpolyt}, it suffices to 
show that for $1 \leq p \leq k$, if 
\begin{equation}\label{eq:R1}
{R}_{u\Pi_p,v\Pi_p}(q)=q{R}_{u\Pi_p t,v\Pi_p t}(q)+(q-1){R}_{u \Pi_p,v \Pi_p t}(q)
\end{equation}
then
\begin{equation}\label{eq:R2}
{R}_{u \Pi_{p-1},v\Pi_{p-1}}(q)=q{R}_{u\Pi_{p-1}t,v\Pi_{p-1}t}(q)+(q-1){R}_{u\Pi_{p-1},v\Pi_{p-1} t}(q).
\end{equation}
By Proposition \ref{prop:R}, we have that  
\[
{R}_{u\Pi_p,v\Pi_p}(q)={R}_{u\Pi_{p-1},v \Pi_{p-1}}(q).
\]
Note that for any $m$, $ts_m=s_mt$. Therefore $u\Pi_p t=u\Pi_{p-1} t s_m$ and $v \Pi_p t=v \Pi_{p-1} t s_m$. This implies that
\[
{R}_{u \Pi_p t,v \Pi_p t}(q)={R}_{u \Pi_{p-1} t,v \Pi_{p-1} t}(q).
\]
Similarly, we have that 
\[
{R}_{u\Pi_p,v\Pi_p t}(q)={R}_{u \Pi_{p-1},v \Pi_{p-1} t}(q).
\]
This shows that \eqref{eq:R1} implies \eqref{eq:R2}.
\end{proof}

\begin{remark}
The above statement and proof hold mutatis mutandis for the $\tilde{R}$-polynomials, which are a renormalization of the $R$-polynomials.
\end{remark}

\begin{example}
Take $u=21345$, $v=53421$ and $t=(13)$. We have
\[
{R}_{u,v}(q)=q^8-4q^7+7q^6-8q^5+8q^4-8q^3+7q^2-4q+1
\]
\[
{R}_{ut,vt}(q)=q^6-4q^5+7q^4-8q^3+7q^2-4q+1
\]
and
\[
{R}_{u,vt}(q)=q^7-4q^6+7q^5-8q^4+8q^3-7q^2+4q-1.
\]
\end{example}

\begin{definition}
A matching of a graph $G=(V,E)$ is an involution $M:V \rightarrow V$ such that $\{v,M(v)\} \in E$ for all $v \in V$. 
\end{definition}

\begin{definition}
Let $P$ be a graded poset. A matching $M$ of the Hasse diagram of $P$ is a \emph{special matching} if for all $x , y \in P$ such that $x \lessdot y$, we have $M(x) = y$ or $M(x) \leq M(y)$.
\end{definition}

It is known that  special matchings can be used to compute $R$-polynomials:

\begin{theorem} \cite[Theorem 7.8]{MR2222360} Let $(W,S)$ be a Coxeter system, 
let $w \in W$, and let $M$ be a special matching of the Hasse diagram of 
the interval $[e,w]$ in Bruhat order. Then
\[
R_{u,w}(q)=q^c R_{M(u),M(w)}(q)+(q^c-1)R_{u,M(w)}(q)
\]
for all $u \leq w$, where $c=1$ if $M(u) \gtrdot u$ and $c=0$ otherwise.
\end{theorem}

One might guess that the Generalized lifting property is compatible with 
the notion of special matching. More precisely, one
might speculate 
that if $[u,v] \subset S_n$ and $t$ is inversion-minimal on $(u,v)$ then there is a special matching $M$ of $[u,v]$ such that $M(u)=ut$ and $M(v)=vt$. 
The following gives an example of this.

\begin{example}
Take $u=143265$ and $v=254163$. Then $t=(36)$ is inversion-minimal on $(u,v)$. 
Suppose that a special matching $M$ of $[u,v]$ (see Figure \ref{interval143265to254163}) satisfies $M(v)=vt$ and $M(u)=ut$.  Then we must have
$M(154263)=153264$ and $M(243165)=245163$. Observe that the result is a multiplication matching.
Similarly, if we take $t=(14)$, another inversion-minimal transposition
on $(u,v)$, we again obtain a multiplication matching.

\begin{center}
\begin{tikzpicture}[scale=.5] 
\node (u) at (0,-6) {\small{$143265$}};

\node (a) at (-4,-2) {\small{$153264$}};

\node (b) at (0,-2) {\small{$243165$}};

\node (c) at (4,-2) {\small{$145263$}};

\node (d) at (-4,2) {\small{$253164$}};

\node (e) at (-0,2) {\small{$154263$}};

\node (f) at (4,2) {\small{$245163$}};

\node (v) at (0,6) {\small{$254163$}};

\draw  (u) -- (a);
\node (x1) at (-3,-4) {\tiny{$(26)$}};
\draw  (u) -- (b);
\node (x2) at (-0.6,-4) {\tiny{$(14)$}};
\draw  [thick] (u) -- (c);
\node (x3) at (3,-4) {\tiny{$(36)$}};

\draw  [thick] (v) -- (d);
\node (x4) at (-3,4) {\tiny{$(36)$}};
\draw  (v) -- (e);
\node (x4) at (-0.6,4) {\tiny{$(14)$}};
\draw  (v) -- (f);
\node (x5) at (3,4) {\tiny{$(23)$}};

\draw  (a) -- (d);
\draw  [thick] (a) -- (e);

\draw  (b) -- (d);

\draw  [thick] (b) -- (f);
\draw  (c) -- (e);
\draw  (c) -- (f);

\end{tikzpicture}
\captionof{figure}{}
\label{interval143265to254163}
\end{center}

\end{example}

The following example shows that it is not the case that an inversion-minimal transposition must be compatible with a special matching. This makes Proposition \ref{Rpolyt} all the more surprising, and shows that it cannot be deduced using special matchings.

\begin{example}
Take $u=1324$ and $v=4312$. Then $t=(24)$ is inversion-minimal on $(u,v)$. Suppose that a special matching $M$ of $[u,v]$ (Figure \ref{interval1323to4312}) satisfies $M(v)=vt$, i.e., sends $4312$ to $4213$. Then
\[
M(4132)=4123, \hspace{.1cm} M(1432)=1423,\hspace{.1cm} M(1342)=1324, \hspace{.1cm} M(3142)=3124, \hspace{.1cm} M(3412)=3214, \hspace{.1cm} M(2413)=2314.
\]
But $M(2314)=2413 \not \geq 1342= M(1324)$, which is a contradiction.

\begin{center}
\begin{tikzpicture}[scale=.5] 
\node (a) at (0,-6) {\small{$1324$}};

\node (a) at (-6,-2) {\small{$1423$}};

\node (b) at (-2,-2) {\small{$1342$}};

\node (c) at (2,-2) {\small{$3124$}};

\node (d) at (6,-2) {\small{$2314$}};

\node (e) at (-8,2) {\small{$1432$}};

\node (f) at (-4,2) {\small{$4123$}};

\node (g) at (0,2) {\small{$2413$}};

\node (h) at (4,2) {\small{$3142$}};

\node (i) at (8,2) {\small{$3214$}};

\node (j) at (-4,6) {\small{$4132$}};

\node (k) at (0,6) {\small{$4213$}};

\node (l) at (4,6) {\small{$3412$}};

\node (v) at (0,10) {\small{$4312$}};

\node (x1) at (0.8,7.7) {$(24)$};
\node (x2) at (-3.5,-4.5) {$(24)$};

\draw  (u) -- (a);
\draw  (u) -- (b);
\draw  (u) -- (c);
\draw  (u) -- (d);
\draw  (a) -- (e);
\draw  (a) -- (f);
\draw  (a) -- (g);
\draw  (b) -- (e);
\draw  (b) -- (h);
\draw  (c) -- (f);
\draw  (c) -- (h);
\draw  (c) -- (i);
\draw  (d) -- (g);
\draw  (d) -- (i);
\draw  (e) -- (j);
\draw  (e) -- (l);
\draw  (f) -- (j);
\draw  (f) -- (k);
\draw  (g) -- (k);
\draw  (g) -- (l);
\draw  (h) -- (j);
\draw  (h) -- (l);
\draw  (i) -- (k);
\draw  (i) -- (l);
\draw  (j) -- (v);
\draw  (k) -- (v);
\draw  (l) -- (v);

\end{tikzpicture}
\captionof{figure}{}
\label{interval1323to4312}
\end{center}

\end{example}

\section{Background on partial flag varieties $G/P$}\label{sec:flag}

\subsection{Preliminaries}\label{sec:preliminaries}
Let $G$ be a semisimple, simply connected linear algebraic group over $\C$ 
split over
$\R$, with split torus $T$.
We identify $G$ (and related spaces)
with their real points and consider them with their real topology.
Let $\mathfrak{t}$ denote the Lie algebra of $T$,
$\mathfrak{t}_{\R}$ denote its real part, and let 
$\mathfrak{t}_{\R}^*$ denote the dual of the torus.
Let $\Phi \subset \mathfrak{t}_{\R}^*$ denote the set of roots,
and choose a system of positive roots $\Phi^+$.  We denote by $B=B^+$ the
Borel subgroup corresponding to $\Phi^+$ and by $U^+$ its unipotent
radical.  We also have the opposite Borel subgroup $B^-$ such that
$B^+ \cap B^- = T$, and its unipotent radical $U^-$. For background
on algebraic groups, see e.g. \cite{Humphreys}.

Let $\Pi = \{\alpha_i \ \vline \ i \in I \} \subset \Phi^+$
denote the simple roots, 
and let $\{\omega_i \ \vline \ i \in I\}$ denote the fundamental weights.
For each $\alpha_i \in \Pi$ there is an associated homomorphism
$\phi_i : \Sl_2 \to G$.
Consider the $1$-parameter subgroups in $G$ (landing in $U^+, U^-$,
and $T$, respectively) defined by
\begin{equation*}
x_i(m) = \phi_i \left(
                   \begin{array}{cc}
                     1 & m \\ 0 & 1\\
                   \end{array} \right) ,\
y_i(m) = \phi_i \left(
                   \begin{array}{cc}
                     1 & 0 \\ m & 1\\
                   \end{array} \right) ,\
\alpha_i^{\vee}(\ell) = \phi_i \left(
                   \begin{array}{cc}
                     \ell & 0 \\ 0 & \ell^{-1}\\
                   \end{array} \right) ,
\end{equation*}
where $m \in \R, \ell \in \R^*, i \in I$.
The datum $(T, B^+, B^-, x_i, y_i; i \in I)$ for $G$ is
called a {\it pinning}.  The standard pinning for
$\Sl_n$ consists of the diagonal, upper-triangular, and lower-triangular
matrices, along with the simple root subgroups
$x_i (m) = I_n + mE_{i,i+1}$ and
$y_i (m) = I_n + mE_{i+1,i}$ where
$I_n$ is the identity matrix and $E_{i,j}$ has a $1$ in
position $(i,j)$ and zeroes elsewhere.

The Weyl group $W = N_G(T) / T$ acts on $\mathfrak{t}^*_{\R}$,
permuting  the roots
$\Phi$.
We set $s_i:= \dot{s_i} T$ where
$\dot{s_i} :=
                 \phi_i \left(
                   \begin{array}{cc}
                     0 & -1 \\ 1 & 0\\
                   \end{array} \right).$
Then any $w \in W$ can be expressed as a product
$w = s_{i_1} s_{i_2} \dots s_{i_m}$ with $\ell(w)$ factors.
This gives $W$ the structure of a Coxeter group; we will
assume some
basic knowledge
of Coxeter systems and Bruhat order as in
\cite{BB}.
We set
$\dot{w} = \dot{s}_{i_1} \dot{s}_{i_2} \dots \dot{s}_{i_m}$.
It is known that $\dot w$ is independent of the reduced expression
chosen.

The \emph{(complete) flag variety} is the homogeneous space
$G/B^+=G/B$. Note that we will frequently use $B$ to denote $B^+$.

We have two opposite Bruhat decompositions of $G/B$:
\begin{equation*}
G/B=\bigsqcup_{v\in W} B \dot v B/B=\bigsqcup_{u\in W}
B^- \dot u B/B.
\end{equation*}
We define the intersection of opposite Bruhat cells
\begin{equation*}
\mathcal R_{u,v}:=(B\dot v B/B)\cap (B^-\dot u B/B),
\end{equation*}
which is nonempty
precisely when $u \leq v$ in Bruhat order, and in that 
case is irreducible of dimension $\ell(v)-\ell(u)$, see \cite{KL}.
The strata $\mathcal R_{u,v}$ are often called \emph{Richardson varieties}.

Let $J \subset I$.  The parabolic subgroup
$W_J \subset W$
corresponds to a parabolic subgroup $P_J$ in $G$
containing $B$.  Namely,
$P_J = \sqcup_{w \in W_J} B \dot{w} B$.
There is a corresponding
\emph{generalized partial flag variety}, which is
the homogeneous space $G/P_J$.

There is a natural projection
from the complete flag variety to
a partial flag variety which takes the form
$\pi = \pi^J: G/B \to G/P_J$, where
$\pi(gB) = gP_J$.

\subsection{Generalized Pl\"ucker coordinates and the 
Gelfand-Serganova stratification of $G/P$}\label{sec:Plucker}

Let $P=P_J$ be a parabolic subgroup of $G$.
In \cite{GelfSerg}, Gelfand and Serganova defined a new
stratification of $G/P$.  In the case that 
$G = \SL_n$ and $P$ is a maximal parabolic subgroup,
their stratification recovers the well-known matroid stratification
of the Grassmannian.

Let $C$ be a Borel subgroup of $G$ containing the maximal torus $T$.
The Schubert cells on $G/P$ associated with $C$ are the orbits of $C$
in $G/P$.  The Schubert cells are in bijection with 
$W^J$, and can be written as $C \dot w P$ where $w\in W^J$.

\begin{definition}
The \emph{Gelfand-Serganova stratification} 
(or \emph{thin cell stratification}) of $G/P$ is the simultaneous
refinement of all the Schubert cell decompositions described above.
The (nonempty) strata in this decomposition are called \emph{Gelfand-Serganova
strata} or \emph{thin cells}.
In other words,  we choose for each Borel 
subgroup $C$ a Schubert cell associated with $C$.  The intersection
of all chosen cells, if it is nonempy, is called a Gelfand-Serganova 
stratum or a thin cell.
\end{definition}

There is another way to think about the Gelfand-Serganova stratification,
using generalized Pl\"ucker coordinates.

Let $J \subset I$ index
the simple roots corresponding to the parabolic subgroup
$P=P_J$, and let ${\rho_J} = \sum_{j \in J} \omega_j$. 
Let $V_{\rho_J}$ be the representation of $G$ with 
highest weight $\rho_j$, and choose a highest weight vector $\eta_J$.
Recall that we have an embedding of the flag variety 
$$G/P \hookrightarrow \mathbb{P} (V_{\rho_J})$$
given by 
$$gP \mapsto g \cdot \eta_J.$$

Let $\mathcal{A}$ be the set of weights of $V_{\rho_J}$ taken
with multiplicitiy.  We choose a weight basis 
$\{e_{\alpha} \vert \alpha \in \mathcal{A}\}$ in $V_{\rho_J}$.
Then any point $X \in G/P$ determines, uniquely up to scalar $d$,
a collection of numbers $p^{\alpha}(X)$,
where 
\begin{equation}\label{gen:Plucker}
X = d \sum_{\alpha \in \mathcal{A}} p^{\alpha}(X) e_{\alpha}.
\end{equation}

Let $W(\rho_J) \subset \mathfrak{t}^*_{\R}$ 
be the orbit of $\rho_J$ under $W$.
Then $W(\rho_J)$ are the \emph{extremal weight vectors},
that is, they lie at the vertices of some convex polytope $\Delta_P$,
and the other elements of $\mathcal{A}$ lie inside of $\Delta_P$,
see \cite{A:82, GelfSerg}.
The extremal weight vectors
can be identified with the set $W/W_J$ of cosets via the map
$w\cdot \rho_J \mapsto wW_J$.

\begin{definition}
Let $X\in G/P$.
The numbers $\{p^{\alpha}(X) \ \vert \ \alpha\in W(\rho_J) \}$,
defined up to scalar, are called the \emph{generalized Pl\"ucker
coordinates} of $X$.  And the \emph{list} of $X$ is the 
subset $$L_X \subset W(\rho_J) = 
\{\alpha \in W(\rho_J) \ \vert \ p^{\alpha(X)} \neq 0\}.$$
\end{definition}

\begin{example}
Let $G = \SL_n$ and $P = \SL_k \times \SL_{n-k}$; note that 
$G/P \cong Gr_{k,n}$.  
Let $V$ denote the $n$-dimensional vector space with standard
basis $e_1,\dots, e_n$. The vector $e_1 \wedge e_2 \wedge \dots \wedge e_k$
is a highest weight vector for the representation $\bigwedge^k V$
of $G$, and we have an embedding
$$G/P \hookrightarrow \mathbb{P} (\bigwedge^k V)$$
given by 
$$gP \mapsto g(e_1 \wedge e_2 \wedge \dots \wedge e_k).$$
Expanding the right-hand side in the natural basis, we get 
$$g(e_1 \wedge e_2 \wedge \dots \wedge e_k) = 
\sum_{I \in {[n] \choose k}} \Delta_I(A) e_{i_1} \wedge \dots \wedge 
e_{i_k},$$
where $I = \{i_1 < \dots < i_k\}$, and $A = \pi_k(g) \in Gr_{k,n}$ is the span
of the leftmost $k$ columns of $A$.
This shows that the generalized Pl\"ucker coordinates 
agree with the Pl\"ucker coordinates from Section \ref{sec:background}
in the case of the Grassmannian.
\end{example}

\begin{theorem}\cite[Theorem 1]{GelfSerg}
Two points $X,Y \in G/P$ lie in the same Gelfand-Serganova stratum if and only if 
they have the same list.
\end{theorem}

\subsection{Total positivity}

We start by reviewing the totally nonnegative part 
$(G/P)_{\geq 0}$ of $G/P$, and Rietsch's cell decomposition of it.
We then relate this cell decomposition to 
the Gelfand-Serganova stratification.

\begin{definition} \cite{Lusztig3}
The totally non-negative
part $U_{\geq 0}^-$ of $U^-$ is defined to be the semigroup in
$U^-$ generated by the $y_i(t)$ for $t \in \R_{\geq 0}$.

The
totally non-negative part 
$(G/B)_{\geq 0}$ of $G/B$ is defined by
\begin{equation*}
(G/B)_{\geq 0} := \overline{ \{y B \ \vline \ y \in U_{\geq 0}^- \} },
\end{equation*}
where the closure is taken inside $G/B$ in its real topology.

The totally non-negative
part 
$(G/P_J)_{\geq 0}$ of $G/P_J$
 is defined to be
$\pi^J ((G/B)_{\geq 0})$.
\end{definition}

Lusztig \cite{Lusztig3, Lusztig2} introduced natural 
decompositions of $(G/B)_{\geq 0}$ and
$(G/P)_{\geq 0}$.

\begin{definition}  \cite{Lusztig3}
For $u,v \in W$ with $u \leq v$, let
\begin{equation*}
\mathcal R_{u,v ; >0} := \mathcal R_{u,v} \cap (G/B)_{\geq 0}.
\end{equation*}
\end{definition}

We write $W^J$ 
(respectively $W^J_{max}$) 
for the set of minimal
(respectively maximal) 
length coset representatives of $W/W_J$.

\begin{definition} \cite{Lusztig2} \label{index}
For $u\in W$ and $v\in W^J$ with $u \leq v$,
define $\mathcal P_{u,v; >0}^J := \pi^J(\mathcal R_{u,v; >0}).$
Here the projection $\pi^J$ is an isomorphism from
$\mathcal R_{u,v; >0}$ to 
$\mathcal P_{u,v; >0}^J$.
\end{definition}

Lusztig conjectured and Rietsch proved \cite{Rietsch1, RietschThesis} 
that $\mathcal R_{u,v}^{>0}$  (and hence $\mathcal P_{u,v;>0}^J$)
is a semi-algebraic cell of dimension
$\ell(v)-\ell(u)$.  
Subsequently Marsh-Rietsch \cite{MR} provided an explicit parameterization of each 
cell.  To state their result, we first review the notion of 
positive distinguished subexpression, as in 
\cite{Deodhar} and \cite{MR}.

Let $\v:= s_{i_1}\dots s_{i_m}$ be a reduced expression for $v\in W$.
A  {\it subexpression} $\u$ of $\v$
is a word obtained from the reduced expression $\v$ by replacing some of
the factors with $1$. For example, consider a reduced expression in 
the symmetric group $S_4$, say $s_3
s_2 s_1 s_3 s_2 s_3$.  Then $1\, s_2\, 1\, 1\, s_2\, s_3$ is a
subexpression of $s_3 s_2 s_1 s_3 s_2 s_3$.
Given a subexpression $\u$,
we set $u_{(k)}$ to be the product of the leftmost $k$
factors of $\u$, if $k \geq 1$, and $u_{(0)}=1$.

\begin{definition}\label{d:Js}\cite{Deodhar, MR}
Given a subexpression $\u$ of  $\v=
s_{i_1} s_{i_2} \dots s_{i_m}$, we define
\begin{align*}
J^{\circ}_\u &:=\{k\in\{1,\dotsc,m\}\ |\  u_{(k-1)}<u_{(k)}\},\\
J^{+}_\u &:=\{k\in\{1,\dotsc,m\}\ |\  u_{(k-1)}=u_{(k)}\},\\
J^{\bullet}_\u &:=\{k\in\{1,\dotsc,m\}\ |\  u_{(k-1)}>u_{(k)}\}.
\end{align*}

The subexpression  $\u$
is called {\it
non-decreasing} if $u_{(j-1)}\le u_{(j)}$ for all $j=1,\dotsc, m$,
e.g.\ if $J^{\bullet}_\u=\emptyset$.
It is called {\it distinguished}
if we have
$u_{(j)}\le u_{(j-1)}\hspace{2pt} s_{i_j}$ for all
$j\in\{1,\dotsc,m\}.$
In other words, if right multiplication by $s_{i_j}$ decreases the
length of $u_{(j-1)}$, then in a distinguished subexpression we
must have
$u_{(j)}=u_{(j-1)}s_{i_j}$.
Finally, 
 $\u$ is called a {\it positive distinguished subexpression}
(or a PDS for short) if
$u_{(j-1)}< u_{(j-1)}s_{i_j}$ for all
$j\in\{1,\dotsc,m\}$.
In other words, it is distinguished and non-decreasing.
\end{definition}

\begin{lemma}\label{l:positive}\cite{MR}
Given $u\le v$
and a reduced expression $\v$ for $v$,
there is a unique PDS $\u_+$ for $u$ contained in $\v$.
\end{lemma}

\begin{theorem}\label{t:parameterization}\cite[Proposition 5.2, Theorem 11.3]{MR}
Choose a reduced expression $\v=s_{i_1} \dots s_{i_m}$ for $v$ with $\ell(v)=m$.
To $u \leq v$ we associate the unique 
PDS 
$\u_+$ for $u$ in $\v$.  Then $J^{\bullet}_{\u^+} = \emptyset$.
We define 
\begin{equation}\label{eq:G}
G_{\u_+,\v}^{>0}:=\left\{g= g_1 g_2\cdots g_m \left
|\begin{array}{ll}
 g_\ell= y_{i_\ell}(p_\ell)& \text{ if $\ell\in J^{+}_\u$,}\\
 g_\ell=\dot s_{i_\ell}& \text{ if $\ell\in J^{\circ}_\u$,}
 \end{array}\right. \right\},
\end{equation}
where each $p_\ell$ ranges over $\R_{>0}$.
Then $G_{\u_+,\v}^{>0} \cong \R_{>0}^{\ell(v)-\ell(u)}$,
and the map $g\mapsto g  B$ defines an isomorphism 
\begin{align*}
G_{\u_+,\v}^{>0}&~\overset\sim\To ~\mathcal R_{u,v}^{>0}.
\end{align*}
\end{theorem}

\begin{remark}\label{rem:partial-parameterization}
Use the notation of Theorem \ref{t:parameterization},
and now assume additionally that $v\in W^J$.
Then Theorem \ref{t:parameterization} and Definition \ref{index} imply 
that the map $g \mapsto g  P_J$ defines an isomorphism 
\begin{align*}
G_{\u_+,\v}^{>0}&~\overset\sim\To ~\mathcal P_{u,v;>0}^J.
\end{align*}
\end{remark}

\begin{definition}
Let $T_{>0}$ denote the positive part of the torus, i.e. the subset of
$T$ generated by all elements
of the form
$\alpha_i^{\vee}(\ell) = \phi_i \left(
                   \begin{array}{cc}
                     \ell & 0 \\ 0 & \ell^{-1}\\
                   \end{array} \right) ,$
where $\ell \in \R_{>0}$.
\end{definition}

\begin{lemma}\label{pos-torus-preserves}
Let $t \in T_{>0}$ and $g P_J \in \mathcal P_{u,v;>0}^J.$
Then $tg B \in \mathcal P_{u,v;>0}^J$.
\end{lemma}

\begin{proof}
We claim that for any $t \in T_{>0}$ and $a\in \R_{>0}$, we have
$t \dot s_i = \dot s_i t'$ for some $t' \in T_{>0}$,
and also $t y_i(a) = y_i(a') t$ for some $a'\in \R_{>0}$.
If we can demonstrate this claim, then the lemma follows
from the parameterization of cells given in Theorem 
\ref{t:parameterization} and Remark \ref{rem:partial-parameterization}:
using the claim, we can simply factor the $t$ all the way to the right 
where it will get absorbed into the group $P_J$.

To prove the first part of the claim, note that since
$\dot s_i$ lies in the normalizer of the torus $N_G(T)$, 
for any $t \in T$ we have that $\dot s_i t \dot s_i^{-1} = t'$ for 
some $t'\in T$.  Moreover, if $t \in T_{>0}$ then also $t' \in T_{>0}$:
one way to see this is to use the fact that $T_{>0}$ is the 
connected component of $T$ containing $1$ \cite[5.10]{Lusztig3}.
Then since $\dot s_i T_{>0} \dot s_i^{-1}$ is also connected and 
contains $1$, its elements must all lie in $T_{>0}$.

To prove the second part of the claim, 
note that by  \cite[1.3 (b)]{Lusztig3}, we have
$t y_i(a) = y_i(\chi'_i(t)^{-1} a)t$ for any $i\in I$, $t\in T$, and 
$a\in \R$, where $\chi'_i:T \to \R^*$ is the simple root corresponding
to $i$.  When $t \in T_{>0}$ and $a\in \R_{>0}$, 
we have $\chi'_i(t)^{-1} a >0$.  
\end{proof}

Rietsch also showed that the closure of each 
cell of $(G/P_J)_{\geq 0}$ is a union of cells, and 
described when one cell of $(G/P_J)_{\geq 0}$
lies in 
the closure of another \cite{RietschClosure}.  Using this description,
it is easy to determine
the set of $0$-cells contained in 
the closure $\overline{\mathcal P_{u,v;>0}^J}$.
\begin{corollary}\label{cor:0-cells}
The $0$-cells in the closure of 
the cell $\mathcal P_{u,v;>0}^J$ of $(G/P_J)_{\geq 0}$
are in bijection with the cosets
$$\{z W_J \ \vert \ u \leq z \leq v \}.$$  More specifically,
those $0$-cells are precisely the cells of the form 
$\mathcal P_{\tilde{z},\tilde{z};>0}^J$, where 
$\tilde{z}$ is the minimal-length coset representative for 
$z$ in $W/W_J$, and $u \leq z \leq v$.
\end{corollary}

\begin{remark}
The $0$-cells in $\overline{\mathcal P_{u,v;>0}^J}$
are precisely the torus fixed points of $G/P_J$ that lie in 
$\overline{\mathcal P_{u,v;>0}^J}$.
\end{remark}

\subsection{Total positivity and canonical bases for simply laced $G$}\label{s:CanonBasis}

\bigskip

Assume that $G$ is simply laced.  Let $\bf U$ be the enveloping
algebra of the Lie algebra of $G$; this can be defined by generators
$e_i, h_i, f_i$ ($i\in I$) and the Serre relations.  For any
dominant weight
$\lambda\in \mathfrak{t}^*_{\R}$,
there is a finite-dimensional simple $\bf U$-module $V(\lambda)$
with a non-zero vector $\eta$ such that $e_i\cdot \eta = 0$
and $h_i \cdot\eta = \lambda(h_i) \eta$ for all $i \in I$.
The
pair $(V(\lambda), \eta)$ is determined up to unique isomorphism.

There is a unique $G$-module structure on $V(\lambda)$ such that
for any $i\in I, a\in \R$ we have
\begin{equation*}
x_i(a) = \exp(a e_i):V(\lambda) \to V(\lambda), \qquad
y_i(a) = \exp(a f_i): V(\lambda) \to V(\lambda).
\end{equation*}
Then
$x_i(a)\cdot\eta = \eta$ for all $i\in I$, $a\in \R$, and
$t\cdot \eta = \lambda(t) \eta$ for all $t \in T$.  Let $\B(\lambda)$
be the canonical basis of $V(\lambda)$ that contains $\eta$
\cite{Lus:CanonBasis}.
We now collect some useful facts about the canonical basis.

\begin{lemma}\label{lem:positive2} \cite[1.7(a)]{Lusztig2}.
For any $w\in W$, the vector $\dot w\cdot \eta$ is the unique element of
$\B(\lambda)$ which lies in the extremal weight space
$V(\lambda)^{w \cdot \lambda}$.  In particular, $\dot w\cdot \eta \in \B(\lambda)$.
\end{lemma}

We define $f_i^{(p)}$ to be $\frac{f_i^p}{p!}$.

\begin{lemma}\label{lem:canonbasis2}
Let $s_{i_1} \dots s_{i_n}$ be a reduced expression for $w\in W$.
Then there exists $a \in \N$ such that
$f_{i_1}^{(a)} \dot s_{i_2} \dot s_{i_3} \dots \dot s_{i_n} \cdot \eta = 
\dot s_{i_1} \dot s_{i_2} \dots \dot s_{i_n}\cdot \eta$. Moreover,
$f_{i_1}^{(a+1)} \dot s_{i_2} \dot s_{i_3} \dots \dot s_{i_n}\cdot \eta = 0$.
\end{lemma}

\begin{proof}
This follows from Lemma \ref{lem:positive2} and properties of the
canonical basis, see e.g.\ the proof of \cite[Proposition 28.1.4]{Lus:Quantum}.
\end{proof}

\subsection{The moment map for $G/P$}

In this section we start by defining the moment
map for $G/P$ and describing some of its properties.  
We then give a result of 
Gelfand-Serganova \cite{GelfSerg} which gives
another description of their stratification of $G/P$ in terms
of the moment map.

Recall the notation of Section \ref{sec:Plucker}.
\begin{definition}
The moment map on $G/P$ is the map $\mu: G/P \to \ttt_{\R}^*$
defined by 
\begin{equation*}
\mu(X) = \frac{\sum_{\alpha \in \mathcal{A}}|p^{\alpha}(X)|^2 \alpha}{\sum_{\alpha \in \mathcal{A}} |p^{\alpha}(X)|^2},
\end{equation*}
where 
$$X = d \sum_{\alpha \in \mathcal{A}} p^{\alpha}(X) e_{\alpha}.$$
\end{definition}

Given $X \in G/P$, let $TX$ denote the orbit of $X$ under the action
of $T$, and $\overline{TX}$ its closure.

Theorem \ref{moment} 
follows from classical work of Atiyah \cite{A:82}
and Guillemin-Sternberg \cite{GuS}.
See also
\cite[Theorem 3.1]{GelfSerg}.

\begin{theorem}\cite[Theorem 3.1]{GelfSerg}\label{moment}
Let $X \in G/P$.  The image $\mu(\overline{TX})$ is a convex
polytope, and $\mu$ induces a one-to-one correspondence between
the set of orbits of $T$ in $\overline{TX}$ and the set of 
faces of the polytope $\mu(\overline{TX})$, whereby a 
$q$-dimensional orbit of $T$ is mapped onto an open $q$-dimensional
face of $\mu(\overline{TX})$.
\end{theorem}

Gelfand and Serganova \cite{GelfSerg} characterized the vertices
of $\mu(\overline{TX})$.

\begin{proposition}\cite[Proposition 5.1]{GelfSerg}\label{prop:moment}
Let $X\in G/P$.  Then  the vertices of $\mu(\overline{TX})$ are the points
$\alpha$ for all $\alpha \in L_X$.
\end{proposition}

\section{Gelfand-Serganova strata, total positivity, and 
Bruhat interval polytopes for $G/P$}\label{sec:genBIP}

In this section we show that each totally positive cell
of $(G/P)_{\geq 0}$ lies in a Gelfand-Serganova stratum, and we 
explicitly determine which one (i.e. we determine the list).\footnote{In
the $G/B$ case, this result was conjectured in 
Rietsch's thesis \cite{RietschThesis}.  Moreover 
Theorem \ref{th:MR} was partially proved in an unpublished
manuscript of Marsh and Rietsch \cite{MR3}.
The second author is grateful to Robert Marsh and Konni Rietsch for 
generously sharing their
ideas.}  We then define a Bruhat interval polytope for $G/P$,
and show that each face of a Bruhat interval polytope is a 
Bruhat interval polytope.  
Our proof of this result on faces 
uses tools from total positivity.  Allen Knutson has
informed us that he has a different proof of this result about faces, using
Frobenius splitting \cite{AK}.

\subsection{Gelfand-Serganova strata and total positivity}

Write $G/P = G/P_J$.  Our goal is to prove the following theorem.

\begin{theorem}\label{th:MR}
Let $u,v \in W$ with $v\in W^J$ and $u \leq v$.
If $X \in \PP_{u,v; >0}^J$, then the list $L_X$ of $X$
is the set 
$\{z \cdot \rho_J \in W(\rho_J) \ \vert \ u \leq z \leq v \}$.
In particular, 
$\PP_{u,v; >0}^J$ is entirely contained in one Gelfand-Serganova stratum.
\end{theorem}

Theorem \ref{th:MR} immediately implies the following.
\begin{corollary}
Suppose that whenever $(u,v) \neq (u',v')$
(where $(u,v)$ and $(u',v')$ index cells of $(G/P)_{\geq 0}$),
we have that 
$\{z \cdot \rho_J \in W(\rho_J) \ \vert \ u \leq z \leq v \}
\neq \{z \cdot \rho_J \in W(\rho_J) \ \vert \ u \leq z \leq v \}$.
Then Rietsch's cell decomposition of $(G/P)_{\geq 0}$ is the restriction
of the Gelfand-Serganova stratification to $(G/P)_{\geq 0}$.
In particular, her cell decompositions of the totally nonnegative parts of 
the complete flag variety $(G/B)_{\geq 0}$ and of the Grassmannian
$(Gr_{k,n})_{\geq 0}$ are the restrictions of the Gelfand-Sergova
stratification to $(G/B)_{\geq 0}$ and $(Gr_{k,n})_{\geq 0}$,
respectively.
\end{corollary}

\begin{remark}
In general, it is possible for two distinct cells to lie in the same  
Gelfand-Serganova stratum.
 For example, let $W = S_4 = S_{\{1,2,3,4\}}$ and $W_J = 
S_{\{1\}} \times S_{\{2,3\}} \times S_{\{4\}}$.  
Let $(u,v) = (e, 4231)$ and let $(u',v') = (1324, 4231)$.  Then 
the minimal-length coset representatives in $W/W_J$ of both 
$\{z \ \vert \ u \leq z \leq v\}$ and 
$\{z \ \vert \ u' \leq z \leq v'\}$ coincide and hence
$\{z \cdot \rho_J \in W(\rho_J) \ \vert \ u \leq z \leq v \}
= \{z \cdot \rho_J \in W(\rho_J) \ \vert \ u \leq z \leq v \}$.
It follows that the cells 
$\PP_{u,v; >0}^J$ and 
$\PP_{u',v'; >0}^J$ both lie in the same Gelfand-Serganova stratum.
\end{remark}

To prove Theorem \ref{th:MR}, we will prove 
Proposition \ref{prop:A} and 
Proposition \ref{prop:B} below.

\begin{proposition}\label{prop:A}
Let $u,v \in W$ with $v\in W^J$ and $u \leq v$.
If $X \in \PP_{u,v; >0}^J$, then 
the list $L_X$ is contained in $\{ z \cdot \rho_J \in W(\rho_J) \ \vert \ u \leq z \leq v \}$.
\end{proposition}

Before proving Proposition \ref{prop:A}, we need a lemma about
the moment map image of 
\emph{positive} torus orbits.

\begin{lemma}\label{lem:postorus}
Let $X \in G/P$.  Recall that $T_{>0}$ denotes the positive part of the torus.
Then 
$\mu({TX}) = \mu({T_{>0}X})$ and 
$\mu(\overline{TX}) = \mu(\overline{T_{>0}X})$.
\end{lemma}

\begin{proof}
Recall that the torus $T$ acts on the highest weight vector $\eta_J$ of
$V_{\rho_J}$ by  $t \eta_J = \rho_J(t) \eta_J$ for all $t\in T$.
So the action of $t\in T$ on $X\in G/P$ will scale the Pl\"ucker coordinates of $X$
by $\rho_J(t)$.

Since the elements $\alpha_j^{\vee}(\ell)$ for $\ell \in \C^*$ generate $T$,
and we can write any $\ell \in \C^*$ in the form 
$r e^{i\theta}$ with $r \in \R_{>0}$ and $\theta \in \R$,
to prove the lemma, it suffices to show that for any 
positive $r$ and real $\theta$,
\begin{equation}\label{identity}
\mu(\alpha_j^{\vee}(r e^{i \theta}) X) = \mu(\alpha_j^{\vee}(r) X).
\end{equation}

First suppose that $e^{i \theta}$ has finite order in the group of complex numbers
of norm $1$.  Then $\alpha_j^{\vee}(e^{i \theta})$ has finite order,
and hence $|\rho_J( \alpha_j^{\vee}(e^{i \theta}))| = 1$.
But now within the group of unit complex numbers, the elements of finite order are
dense.  Therefore for any unit complex number $e^{i \theta}$, 
we have $|\rho_J( \alpha_j^{\vee}(e^{i \theta}))| = 1$.

Now note that since $\rho_J$ and $\phi_j$ are homomorphisms, we have
$$|\rho_J(\alpha_j^{\vee}(r e^{i \theta}))| = 
|\rho_J(\alpha_j^{\vee}(r)) \rho_J( \alpha_j^{\vee}(e^{i \theta}))| = 
|\rho_J(\alpha_j^{\vee}(r))|.$$
And since the moment map depends only on the absolute value of the 
Pl\"ucker coordinates, it follows that 
\eqref{identity} holds.
\end{proof}

We now turn to the proof of Proposition \ref{prop:A}.
\begin{proof}[Proof of Proposition \ref{prop:A}]
Consider $ z \cdot \rho_J \in L_X$.  
Since the extremal weight vectors are in bijection with cosets
$W/W_J$, we may assume that $z \in W^J$.
Proposition \ref{prop:moment} implies that $ z \cdot \rho_J$ is 
a vertex of $\mu(\overline{TX})$, and 
Lemma \ref{lem:postorus} therefore implies that $ z \cdot \rho_J \in
\mu(\overline{T_{>0}X})$.  Choose
$X'\in \overline{T_{>0}X}$ such that $\mu(X')= z \cdot \rho_J$.
Since $z \cdot \rho_J$ is a vertex of $\mu(\overline{TX})$,
Theorem \ref{moment} implies that $z \cdot \rho_J$
is the image of a torus fixed point of $\overline{TX}$.
Therefore $X'$ is a torus fixed point of $G/P$ and must necessarily
be the point $\dot z P$.

By Lemma \ref{pos-torus-preserves}, $\overline{T_{>0}X} \subset
\overline{\PP_{u,v; >0}^J}$.
Therefore $X' = \dot zP \in \overline{\PP_{u,v; >0}^J}$.  
It follows that $\dot zP$ is a $0$-cell in the closure of 
$\PP_{u,v;>0}^J$, so by Corollary \ref{cor:0-cells},
we must have $u \leq z \leq v$.
\end{proof}

\begin{proposition}\label{prop:B}
Let $u,v \in W$ with $v\in W^J$ and $u \leq v$.
If $X \in \PP_{u,v; >0}^J$, then 
the list $L_X$ contains $\{ z \cdot \rho_J \in W(\eta_J) \ \vert \ u \leq z \leq v \}$.
\end{proposition}

To prove Proposition \ref{prop:B}, we 
will need Proposition \ref{prop:embedding} below, 
which follows from 
\cite[Lemma 6.1]{RW-CW}.
In fact the statements in \cite[Section 6]{RW-CW} used the 
$\rho$-representation $V_{\rho}$ of $G$, but the arguments apply unchanged
when one uses $V_{\rho_J}$ in place of $V_{\rho}$.

\begin{proposition}\label{prop:embedding}
Consider $G/P_J$ where $G$ is simply laced. 
Let $\eta_J$ be a highest weight vector of $V_{\rho_J}$
and let $\mathcal{B} = \mathcal{B}(\rho_J)$ be the canonical basis
of $V_{\rho_J}$
\cite{Lus:CanonBasis} which contains $\eta_J$.
Consider the embedding
$$G/P_J  \hookrightarrow \mathbb{P}(V_{\rho_J})$$  where
$$gP_J \mapsto \langle g \cdot \eta_J \rangle.$$
Then if $gP_J \in (G/P_J)_{\geq 0}$, 
the line $\langle g \cdot \eta_J \rangle$, when expanded in 
$\mathcal{B}$, has non-negative coefficients.
Moreover, the set of coefficients which are positive (respectively, zero)
depends only on which cell of $(G/P_J)_{\geq 0}$ the element 
$gP_J$ lies in.
\end{proposition}

We are now ready to prove Proposition \ref{prop:B}.  We will first
prove it in the simply-laced case, using properties of the canonical 
basis, following Marsh and Rietsch,
and then prove it in the general case, using folding.

\begin{proof}\cite{MR3}[Proof of Proposition \ref{prop:B} when $G$ is simply-laced.]
Let $z \in [u,v]$.  We will use induction on $\ell(v)$ to show that 
$ z \cdot \rho_J$ is in the list $L_X$.

By Remark \ref{rem:partial-parameterization}, we can write 
$X = g P_J$ for $g = g_1 \dots g_n \in G_{\u_+,\v}^{>0}$,
where $\v = s_{i_1} \dots s_{i_n}$ is a reduced expression of $v$.
Define $\v' = s_{i_2} \dots s_{i_n}$, 
$g' = g_2 \dots g_n$, and  $X' = g' P_J$.
We need to consider two cases: that $g_1 = y_{i_1}(p_1)$,
and $g_1 = \dot s_{i_1}$.

In the first case, we have $X' \in \PP_{u,v';>0}^J$.
So the induction hypothesis implies that 
$$L_{X'} = \{ z \cdot \rho_J \ \vert \ u \leq z \leq v' \}.$$
Here we must have $L_{X'} \subset L_X$, since 
$\PP_{u,v';>0}^J \subset \overline{\PP_{u,v;>0}^J}$.
So we are done if $u \leq z \leq v'$.
Otherwise, $u \leq z \leq v$ but $z \nleq v'$.
Then any subexpression for $z$ within $\v$ must use the $s_{i_1}$,
and so $u \leq s_{i_1} z \leq v'$.
And now by induction, we have $ s_{i_1} z \cdot \rho_J \in L_{X'}$.

By Proposition \ref{prop:embedding},
the line  $\langle X' \cdot \eta_j \rangle = \langle g' \cdot \eta_J \rangle$ is spanned by a vector 
$\xi$ which is a non-negative linear combination of 
canonical basis elements.
Since $s_{i_1}  z \cdot \rho_J \in L_{X'}$, 
we have that $\xi = c \dot s_{i_1} \dot z \cdot \eta_J + 
\text{other terms},$ where $c$ is positive.
By Lemma \ref{lem:canonbasis2},
when we apply $y_{i_1}$ to $\xi$ we see that
$\langle X \cdot \eta_J \rangle = \langle c' \dot z \cdot \eta_J + \text{other terms}\rangle,$ where $c' \neq 0$.
Therefore $z \cdot \rho_J \in L_X$.

In the second case, we have that $g_1 = \dot s_{i_1}$,
so $u':=s_{i_1} u \lessdot u$ and 
$v':=s_{i_1} v \lessdot v$.  By the induction hypothesis,
$$L_{X'} = \{ z \cdot \rho_J \ \vert \ u' \leq z \leq v'\}.$$
Consider again $u \leq z \leq v$.
Since the positive subexpression $\u_+$ for $u$ in $\v$
begins with $s_{i_1}$, we must have 
$u \nleq v'$.  
But then $z \nleq v'$.

Now $z \leq v$ and $z \nleq v'$ implies that any reduced expression
for $z$ in $\v$ must use the $s_{i_1}$.  So if we let 
$z':= s_{i_1} z$, then 
$u' \leq z' \leq v'$.
Therefore by the induction hypothesis, $ z' \cdot \rho_J \in L_{X'}$,
i.e.
$\langle X' \cdot \eta_J \rangle = 
\langle c \dot z' \cdot \eta_J + \text{other terms} \rangle$
where $c \neq 0$.
But now $\langle X \cdot \eta_J \rangle = 
\langle \dot s_{i_1} X' \cdot \eta_J \rangle = 
\langle c \dot s_{i_1} \dot z' \cdot \eta_J + \text{other terms} \rangle = 
\langle c \dot z \cdot \eta_J + \text{other terms} \rangle.$
Therefore $z \cdot \rho_J \in L_X$.
\end{proof}

Before proving Proposition \ref{prop:B} in the general case, 
we give a brief overview
of  how one can view each $\overline G$ which is not simply laced
in terms of a simply laced group $G$ by ``folding."
For a detailed explanation of how folding works, see \cite{Stembridge}.

If $\overline G$ is not simply laced, then one can
construct a simply laced group $G$ and an automorphism
$\tau$ of $G$ defined over $\R$, such that there is
an isomorphism, also defined over $\R$, between $\overline G$ and
 the fixed point subset $G^\tau$ of $G$. Moreover
 the groups $\overline G$ and $G$ have compatible pinnings.  Explicitly
 we have the following.

Let $G$ be simply connected and simply laced.  Choose a pinning
$(T, B^+, B^-,  x_i, y_i, i\in  I)$ of $G$.  Here
 $I$ may be identified with the vertex set of the Dynkin
diagram of $G$.
Let $\sigma$ be a permutation of $I$ preserving
connected components of the Dynkin diagram, such that,
if $j$ and $j'$ lie in  the same orbit under $\sigma$
then they are {\it not} connected by an edge.
Then $\sigma$ determines an automorphism $\tau$ of $G$
such that
$\tau(T)= T$; and 
for all
$i\in  I$ and $m\in \R$, we have
$\tau(x_i(m))=x_{\sigma(i)}(m)$ and $\tau(y_i(m))=y_{\sigma(i)}(m)$. 
In particular $\tau$ also preserves $ B^+, B^-$. Let $\overline I$ denote
the set of $\sigma$-orbits in $I$, and for $\overline i\in \overline I$, let
\begin{equation*}
x_{\overline i}(m):=\prod_{i\in \overline i}\ x_i(m), \text{ and }
y_{\overline i}(m):=\prod_{i\in \overline i}\ y_i(m). 
\end{equation*}
We also let $s_{\overline i} = \prod_{i\in \overline i} s_i$, 
and $\alpha_{\overline i} = \sum_{i \in \overline i} \alpha_i$,
where $\{\alpha_i \ \vert \ i\in I\}$ is the set of simple roots for $G$.

Then the fixed
point group $G^{\tau}$ is a simply connected algebraic group with  pinning
$({T}^\tau, B^{+ \tau}, B^{-\, \tau},  x_{\overline i}, y_{\overline i}, \overline i\in \overline I)$.
There exists, and we choose, $G$ and $\tau$ such that  $G^{\tau}$
is isomorphic to our group $\overline G$ via an isomorphism compatible with the pinnings.
The set $\{\alpha_{\overline i} \ \vert \ \overline i \in \overline I\}$
is the set of simple roots  for $\overline G$,
and $\overline W:= \langle s_{\overline i} \ \vert \ \overline i \in \overline I \rangle$
is the Weyl group for $\overline G$.  Note that 
$\overline W \subset W$, where $W$ is the Weyl group for $G$.
Moreover, any 
reduced expression $\overline {\mathbf v}=(\overline i_1,\overline i_2,\dotsc, \overline i_m)$ in 
$\overline W$ gives rise to a reduced
expression ${\mathbf v}$ in $W$ of length
$\sum_{k=1}^m |\overline i_k|$, which is determined
uniquely up to commuting elements \cite[Prop. 3.3]{Nanba}.
To a subexpression $\overline {\mathbf u}$ of $\overline {\mathbf v}$
we can then associate a unique subexpression
${\mathbf u}$ of ${\mathbf v}$ in
the obvious way.

\begin{proof}[Proof of Proposition \ref{prop:B} in the general case.]
Let $\overline G$ 
be a group which is not simply laced and use all the notation above.
We have that  $\overline G$ is isomorphic to  $G^{\tau}$
via an isomorphism compatible with the pinnings.
Let 
$\overline P_{\overline J}$ be the parabolic subgroup 
of $\overline G$ determined by the subset $\overline J \subset \overline I$.
This gives rise to a subset $J \subset I$ defined by 
$$J = \bigcup_{\overline i \in \overline J} \overline i.$$

Now note that 
$$\rho_{\overline J} = \sum_{\overline i \in \overline J} \omega_{\overline i} = 
\sum_{i \in  J} \omega_{i} = \rho_{J}.$$  Therefore 
the highest weight vector $\eta_J$ of the $G$-representation
$V_{\rho_J}$ can also be viewed as a highest weight vector
for the $\overline G$-representation $V_{\rho_{\overline J}}$.

It follows from  \cite[Lemma 6.3]{RW-CW} that
we have an inclusion $\overline G^{>0}_{\overline \u^+, \overline \v} \subset
G^{>0}_{\u^+, \v}$.  
Therefore we have an embedding
\begin{equation}\label{eq:embedding}
\mathcal P^{\overline J}_{u,v;>0}\hookrightarrow {\mathcal P}^J_{u,v;>0}
\hookrightarrow \mathbb P(V_{\rho_J})
\end{equation}
defined by 
$$g P_{\overline J} \mapsto g P_{J} \mapsto g \cdot \eta_J,$$
where the $u$ and $v$ in 
$\mathcal P^{\overline J}_{u,v;>0}$ are viewed as elements of 
$\overline W \subset W$, while the $u$ and $v$ in 
${\mathcal P}^J_{u,v;>0}$ are viewed as elements of $W$.

Let $z \in [u,v] \subset  \overline W \subset W$.
Thanks to the embedding \eqref{eq:embedding}
and the fact that Proposition \ref{prop:B} holds in the simply laced case,
we have that for any $X = gP_{\overline J} \in 
\mathcal P^{\overline J}_{u,v;>0}$, 
$\langle g \cdot \eta_J \rangle = 
\langle c z \cdot \eta_J + \text{ other terms } \rangle$,
where $c >0$.  
Therefore $z \cdot \rho_J$ is in the list $L_X$.
\end{proof}

\subsection{Bruhat interval polytopes for $G/P$}

\begin{definition}
Choose a partial flag variety $G/P = G/P_J$.
Recall that $\rho_J$ is the sum of fundamental weights 
$\sum_{j\in J} \omega_j$.
Let $u,v\in W$ with $v\in W^J$ and $u \leq v$.
The \emph{Bruhat interval polytope} $\mathsf{Q}_{u,v}^J$ for $G/P$
is the convex hull 
$$\conv\{z \cdot \rho_J \ \vert u \leq z \leq v\} \subset \mathfrak{t}^*_{\R}$$
\end{definition}

\begin{lemma}\label{lem:torusorbit}
For any $X \in \mathcal{P}_{u,v; >0}^J$, 
$\mu(TX) = \Int (\mathsf{Q}_{u,v}^J)$ and 
$\mu(\overline{TX}) = \mathsf{Q}_{u,v}^J$, where 
$\Int$ denotes the interior.
\end{lemma}

\begin{proof}
By Theorem \ref{th:MR}, the list of $X$ is precisely the set
$\{z \cdot \rho_J \ \vert u \leq z \leq v\}.$
And by Proposition \ref{prop:moment}, the vertices of 
$\mu(\overline{TX})$ are precisely the elements of the list.
Therefore $\mu(\overline{TX}) = \mathsf{Q}_{u,v}^J$.
Finally, Theorem \ref{moment} implies that 
$\mu(TX)$ maps onto the interior of $\mathsf{Q}_{u,v}^J$.
\end{proof}

\begin{proposition}\label{prop:closure}
We have that $\mu(\PP_{u,v;>0}^J) = \Int(\mathsf{Q}_{u,v}^J)$ and 
$\mu(\overline{\PP_{u,v; >0}^J})  =
\mathsf{Q}_{u,v}^J$.
\end{proposition}
\begin{proof}
By Lemma \ref{lem:postorus} and Lemma \ref{lem:torusorbit}, 
for each $X \in \PP_{u,v;>0}^J$, we have
$\mu(T_{>0} X) = \Int(\mathsf{Q}_{u,v}^J)$.
Now using Lemma \ref{pos-torus-preserves}, we have that
$T_{>0} X \subset \PP_{u,v;>0}^J$.  It follows that 
$\mu(\PP_{u,v;>0}^J) = \Int(\mathsf{Q}_{u,v}^J)$.

To prove the second statement of the proposition,
note that since $\mu$ is continuous,
$\mu(\overline{\PP_{u,v;>0}^J}) \subset \overline{\mu(\PP_{u,v;>0}^J)}$,
and hence $\mu(\overline{\PP_{u,v;>0}^J}) \subset \mathsf{Q}_{u,v}^J$.
But now by Lemma \ref{lem:torusorbit}, 
for any $X \in \PP_{u,v;>0}^J$,
$\mu(\overline{TX}) = 
\mathsf{Q}_{u,v}^J$.  Since $\overline{TX} \subset
\overline{\PP_{u,v;>0}^J}$, we obtain
$\mu(\overline{\PP_{u,v;>0}^J}) = \mathsf{Q}_{u,v}^J$.
\end{proof}

\begin{remark}\label{rem:Richardson}
Forgetting about total positivity, one 
can also consider the moment map images of 
Richardson varieties $\mathcal{R}_{u,v}$ and 
projected Richardson varieties $\PP_{u,v}^J = \pi_J(\mathcal{R}_{u,v})$,
for $u \leq v$. 
Using Proposition \ref{prop:closure} and the fact that the torus
fixed points of ${\mathcal{R}_{u,v;>0}}$ and $\mathcal{R}_{u,v}$
agree, 
it follows that 
$\mu(\overline{\mathcal{R}_{u,v}})  =
\mathsf{Q}_{u,v}$.  We similarly have that for $u \leq v$ with 
$v\in W^J$,
$\mu(\overline{\PP_{u,v}^J})  =
\mathsf{Q}_{u,v}^J$.
\end{remark}

Using Remark \ref{rem:Richardson}, we can prove the following.
\begin{proposition}\label{prop:toric}
The Richardson variety 
$\mathcal{R}_{u,v}$ is a toric variety if and only if
$\dim \mathsf{Q}_{u,v} = \ell(v)-\ell(u)$.  Similarly,
if $u \leq v$ and $v\in W^J$, the projected Richardson variety
$\PP_{u,v}^J$ is a toric variety if and only if 
$\dim \mathsf{Q}_{u,v}^J = \ell(v)-\ell(u)$.
\end{proposition}

\begin{proof}
The Richardson variety $\mathcal{R}_{u,v}$ is a toric variety 
if and only if it contains a dense torus.  By Theorem
\ref{moment}, 
$\mathcal{R}_{u,v}$ contains a dense torus if and only if 
$\dim \mu(\overline{\mathcal{R}_{u,v}}) = \dim 
\mathcal{R}_{u,v}$.  
By Remark \ref{rem:Richardson},
$\mu(\overline{\mathcal{R}_{u,v}}) = \mathsf{Q}_{u,v}$.
Also, by \cite{KL}, we have  that 
$\dim \mathcal{R}_{u,v} = \ell(v)-\ell(u)$.
Therefore $\mathcal{R}_{u,v}$ is a toric variety if and only if 
$\dim \mathsf{Q}_{u,v} = \ell(v)-\ell(u)$.
The second statement of the proposition
follows from the first, using the fact that the projection
map $\pi_J$ is an isomorphism from 
$\mathcal{R}_{u,v}$ to $\PP_{u,v}^J$
whenever $v \in W^J$.
\end{proof}

We are now ready to prove Theorem \ref{th:face}.

\begin{theorem}\label{th:face}
The face of a Bruhat interval polytope for $G/P$ 
is a Bruhat interval polytope for $G/P$.
\end{theorem}

\begin{proof}
Consider a Bruhat interval polytope $\mathsf{Q}_{u,v}^J$ for $G/P_J = G/P$.
By Lemma \ref{lem:torusorbit},
we can express $\mathsf{Q}_{u,v}^J = \mu(\overline{TX})$ for 
some  $X \in \PP_{u,v; >0}^J$.

Now let $F$ be a face of $\mathsf{Q}_{u,v}^J$.
By Theorem \ref{moment}, the interior of $F$
is the image of a $T$-orbit of some point 
$Y \in \overline{TX}$.  By Lemma \ref{lem:postorus},
the interior of $F$ is therefore also the image of a 
$T_{>0}$-orbit of some point $Y \in \overline{T_{>0}X}$.
Therefore $Y \in \overline{\PP_{u,v;>0}^J}$, 
and hence $Y$ lies in some cell $\PP_{a,b;>0}^J$ 
with $a,b\in W$, $b\in W^J$, $a \leq b$.

But then $F = \mu(\overline{TY})$ for 
$Y \in {\PP_{a,b;>0}^J}$ and hence
by Lemma \ref{lem:torusorbit},
$F$ is a Bruhat interval polytope.
\end{proof}

\begin{corollary}
Every edge of a Bruhat interval polytope corresponds to a cover
relation in the (strong) Bruhat order.
\end{corollary}

\begin{proof}
Every edge is itself a face of the polytope,
so by Theorem \ref{th:face}, it must come from an interval in Bruhat
order.  Since the elements of $W(\rho_J)$ are the vertices of a polytope,
none lies in the convex hull of any of the others.  So a Bruhat interval
polytope with precisely two vertices must come from a cover relation in 
Bruhat order.
\end{proof}

\begin{example}
When $G = \Sl_n$ and $P=B$, a Bruhat interval polytope for $G/P$
is precisely a Bruhat interval polytope as defined in 
Definition \ref{def:BIP}.
\end{example}

\begin{example}
When $G = \Sl_n$ and $P$ is a maximal parabolic subgroup, $G/P$
is a Grassmannian, $(G/P)_{\geq 0}$ is the totally non-negative
part of the Grassmannian, and the cells $\mathcal{P}_{u,v;>0}^J$
(for $u \leq v$, $u\in S_n$, and $v\in W^J$) are called 
\emph{positroid cells}.  In this case the moment map images of closures
of torus orbits are a special family of matroid polytopes called 
\emph{positroid polytopes}.  These polytopes were studied
in \cite{ARW}; in particular, it was shown there (by a different method)
that a face of a positroid polytope is a positroid polytope.
\end{example}

\bibliographystyle{alpha}
\bibliography{bibliography}



\end{document}